\documentclass[a4paper,10pt]{amsart}
\usepackage{amsmath,amsthm,amsfonts,amssymb}

\usepackage[english]{babel}

\usepackage{tikz-cd, tikz}

\usepackage{setspace}
\usepackage{booktabs}
\usepackage{longtable}

\usepackage{mathtools}
\mathtoolsset{showonlyrefs=true}
\usepackage{bm}
\usepackage{bbm}
\usepackage{tensor}
\usepackage{color,soul}%

\usepackage{array}
\usepackage{caption}

\usepackage[colorlinks,hypertexnames=false]{hyperref}
\hypersetup{citecolor = {blue},
            linkcolor = {red}
}

\theoremstyle{plain}
\newtheorem{theorem}{Theorem}[section]
\newtheorem{lemma}[theorem]{Lemma}
\newtheorem{corollary}[theorem]{Corollary}
\newtheorem{proposition}[theorem]{Proposition}

\theoremstyle{definition}
\newtheorem{definition}[theorem]{Definition}
\newtheorem{definition-theorem}[theorem]{Definition-Theorem}
\newtheorem{example}[theorem]{Example}
\theoremstyle{remark}
\newtheorem{remark}[theorem]{Remark}

\def\Im{\mathrm{Im}}
\def\Id{\mathrm{Id}}

\def\ch{\mathrm{ch}}

\def\cF{\mathcal{F}}
\def\cE{\mathcal{E}}
\def\cO{\mathcal{O}}

\def\cG{\mathcal{G}}

\def\cT{\mathcal{T}}
\def\cK{\mathcal{K}}

\def\RR{\mathbb{R}}

\def\PP{\mathbb{P}}

\def\del{\partial}

\def\R{\mathbb{R}}
\def\Q{\mathbb{Q}}
\def\Z{\mathbb{Z}}
\def\N{\mathbb{N}}

\def\P{\mathbb{P}}

\def\C{\mathbb{C}}

\def\Om{\Omega}

\def\>{\rangle}

\def\<{\langle}
\def\>{\rangle}

\def\Hom{\mathrm{Hom}}
\def\End{\mathrm{End}}
\def\Spec{\mathrm{Spec}}

\def\rk{\mathrm{rk}}

\def\deg{\mathrm{deg}}
\def\cl{\mathrm{cl}}

\def\codim{\mathrm{codim}}

\def\arccot{\mathrm{arccot}}
\def\Pic{\mathrm{Pic}}
\def\rank{\mathrm{rank}}
\def\dim{\mathrm{dim}}

\def\Eff{\mathrm{Eff}}

\def\Cone{\mathrm{Cone}}

\usepackage[baseline]{euflag}%
\newcommand{\m}[1]{\mathcal{#1}}%
\newcommand{\mrm}[1]{\mathrm{#1}}%

\newcommand{\Alt}{\mathcal{A}}
\DeclareMathOperator{\Tr}{Tr}
\DeclareMathOperator{\range}{ran}
\newcommand{\delbar}{\bar{\partial}}%
\newcommand{\dd}{\mathrm{d}}%
\newcommand{\I}{\mkern1mu\mathrm{i}\mkern1mu}
\newcommand{\e}{\mathrm{e}}
\newcommand{\Todd}{\mathrm{Td}}
\newcommand{\chern}{\mathrm{c}}
\newcommand{\Chern}{\mathrm{ch}}
\newcommand{\BC}{\mathrm{BC}}
\newcommand{\curvform}{\mathcal{F}}
\newcommand{\Zch}{Z}
\newcommand{\sym}{\mathrm{sym}}
\newcommand{\Herm}{\mathcal{H}}
\newcommand{\HermPos}{\mathcal{H}^+}
\newcommand{\Holstr}{\mathcal{A}}
\newcommand{\PLagr}{\mathcal{P}}
\DeclarePairedDelimiter\norm{\lVert}{\rVert}
\DeclarePairedDelimiter{\set}{\{}{\}}
\newcommand{\suchthat}{\mathrel{}\mathclose{}\middle|\mathopen{}\mathrel{}}

\renewcommand{\gg}{>\!\!>}

\allowdisplaybreaks
 
\numberwithin{equation}{section}

\title[P-stability, P-positivity, equivariance and blow-ups]
{Polynomial stability conditions for vector bundles:\\ positivity, equivariance and blow-ups}

\hypersetup{pdftitle={P-stability, equivariance and blow-ups}}

\author[R. Delloque]{R\'emi Delloque}
\address{R\'emi Delloque, Univ Brest, UMR CNRS 6205, Laboratoire de Math\'ematiques de Bretagne Atlantique, France}
\email{remi.delloque@univ-brest.fr}

\author[A. Napame]{Achim Napame}
\address{Achim Napame, Instituto de Matem{\'a}tica, Estat{\'i}stica e Computa{\c c}{\~a}o Cient{\'i}fica - UNICAMP, Rua S{\'e}rgio Buarque de Holanda 651, 13083-970 Campinas-SP, Brazil}
\email{achim@unicamp.br}

\author[C. Scarpa]{Carlo Scarpa}
\address{Carlo Scarpa, Universit\'e Claude Bernard Lyon 1, Institut Camille Jordan, 21 av. Claude Bernard, 69100 Villeurbanne, France}
\email{scarpa@math.univ-lyon1.fr}

\author[C. Tipler]{Carl Tipler}
\address{Carl Tipler, Univ. Brest, UMR CNRS 6205, Laboratoire de Math\'ematiques de Bretagne Atlantique, France}
\email{carl.tipler@univ-brest.fr}

\begin{document}

\begin{abstract}
    We introduce the notion of $P$-critical connections for hermitian holomorphic vector bundles over compact balanced manifolds: integrable hermitian connections whose curvature solve a polynomial equation. Such connections include HYM and dHYM connections, as well as solutions to higher rank Monge--Amp\`ere or $J$-equations, and are a slight generalisation of Dervan--McCarthy--Sektnan's $Z$-critical connections motivated by Bayer's polynomial Bridgeland stability conditions. The associated equations come with a moment map interpretation, and we provide numerical conditions that are expected to characterise existence of solutions in suitable cases: $P$-positivity and $P$-stability. We then provide some devices to check those numerical conditions in practice. First, we observe that $P$-positivity is equivalent to its equivariant version over $T$-varieties.  In the toric case, we thus obtain an explicit finite set of subvarieties to test $P$-positivity on, independently on the choice of the polynomial equation. We also introduce equivariant $P$-stability and discuss its relation to $P$-stability. Secondly, we show that a uniform version of $P$-positivity is preserved by pulling back along a blow-up of points. We apply those results to some examples, such as blow-ups of Hirzebruch surfaces, or a Fano $3$-fold.
\end{abstract}

\maketitle

\setcounter{tocdepth}{1}
\tableofcontents

\section{Introduction}\label{sec:intro}

Given a simple holomorphic vector bundle on a compact K\"ahler manifold, the famous Hitchin-Kobayashi correspondence \cite{Donaldson85,uhlenbeckyau86} gives an equivalence between the existence of Hermitian metrics whose Chern connection solves the Hermitian Yang-Mills equation, and the Mumford-Takemoto stability of the bundle. The far-reaching consequences of this correspondence include tools for studying moduli problems of vector bundles, hyperk\"ahler varieties, and the topology of complex surfaces, among others.

On the algebraic side of this correspondence, it has long been clear that other stability conditions might be better suited to the study of moduli problems, and wall-crossing phenomena have rightly acquired great prominence in the subject, given the connections with theoretical physics. Indeed, a wealth of algebraic stability conditions for sheaves are systematically studied in the theory of \emph{Bridgeland stability conditions} introduced in \cite{Bridgeland}.

It is natural to ask whether a holomorphic vector bundle satisfying one of these stability conditions admits a connection satisfying some special curvature property, by analogy with the Kobayashi-Hitchin correspondence. Dervan, McCarthy and Sektnan introduced in \cite{DMS} a class of partial differential equations, called \emph{$\Zch$-critical equations}, each corresponding to a \emph{polynomial central charge} (as defined by Bayer~\cite{Bayer}) on the category of coherent sheaves on a projective manifold, and whose existence of solutions is conjecturally related to the corresponding Bridgeland stability condition.

In this work, we slightly expand the approach of \cite{DMS} to consider a larger class of equations and stability conditions. The difference between the two approaches is mainly cosmetic, as most of the results of \cite{DMS} apply to our case. The formalism we are about to introduce makes several results on the $\Zch$-critical equations more transparent, such as the positivity and stability results of \cite{KScarpa} that we recall below.

\subsection{Polynomial curvature equations}
\label{sec:equations}
First, some notation. Let $E\to X$ be a holomorphic vector bundle over the compact complex manifold $X$ of dimension $n$. Given any affine connection $\nabla$ on $E$, we denote by $F(\nabla)\in\Alt^{1,1}(\End E)$ the curvature of $\nabla$. It will also be convenient to denote by $\curvform(\nabla)\coloneqq\frac{\I}{2\pi}F(\nabla)$ the \emph{normalised curvature form} of $\nabla$, which is a real $\End(E)$-valued $(1,1)$-form on $X$. Given a Hermitian metric on $E$, $h\in\Herm(E)$, we will also denote by $\curvform(h)$ the normalised curvature form of the Chern connection defined by $h$ and the holomorphic structure of $E$. The normalisation is chosen in such a way that the closed differential form
\begin{equation}
\Tr\left(\e^{\curvform(h)}\right)\coloneqq \sum_{j\geq 0}\frac{1}{j!}\Tr\left(\curvform(h)^j\right)
\end{equation}
represents the Chern character $\Chern(E)\in H^\bullet(X)$. Here and in what follows, $\curvform(h)^j$ is the $j$th wedge product $\curvform(h)\wedge\dots\wedge\curvform(h)$, and $\curvform(h)^0=\Id_E$ by convention.

\smallskip

Given closed differential forms $(\gamma_0,\dots,\gamma_n)$ on $X$ with $\gamma_j\in\Alt^{n-j,n-j}(X,\R)$ we consider this \emph{polynomial equation} for the curvature form of a connection $\nabla$ on $E$
\begin{equation}\label{eq:polynomial_eq}
    \sum_{j=0}^n\gamma_j\wedge\frac{1}{j!}\curvform(\nabla)^j=0.
\end{equation}
Note in particular that the leading order coefficient $\gamma_n$ is a constant, while $\gamma_0$ is a top-degree form on $X$. Formally, we write~\eqref{eq:polynomial_eq} as $P(\curvform(\nabla))=0$, where
\begin{equation}\label{eq:polynomial}
    P(x)=\sum_{j=0}^n\gamma_j\wedge\frac{1}{j!}x^j\in\Alt^{\bullet}(X,\RR)[x]
\end{equation}
is a polynomial whose coefficients are closed differential forms. Equation~\eqref{eq:polynomial} is a version, on bundles of arbitrary rank, of the PDEs considered in \cite{FangMa_rank1}: we recover their setting when $\rk(E)=1$.

We will often refer to~\eqref{eq:polynomial_eq} as the \emph{$P$-critical equation}, and its solutions are called \emph{$P$-critical connections}. If $\nabla$ is the Chern connection of $h\in\Herm(E)$, we will also say that $h$ is a \emph{$P$-critical metric}. The $P$-critical equation~\eqref{eq:polynomial_eq} can be interpreted as a moment map equation, see Proposition~\ref{prop:momentmap}, and is the Euler-Lagrange equation of a natural generalisation of Donaldson's functional~\cite{Donaldson85}.
\begin{proposition}\label{prop:Lagrangian}
    Having fixed any $h_0\in\Herm(E)$, the $P$-critical metrics on $E$ are the critical points of the functional
    \begin{equation}
        \PLagr(h,h_0)\coloneqq 2\pi\int_X\left(\textstyle\sum_{j\geq 0}\gamma_j\right)\wedge\BC(h,h_0),
    \end{equation}
    where $\BC(h,h_0)$ is the Bott-Chern representative of the total Chern-Simons class.
\end{proposition}
Using this Lagrangian, we will show a local uniqueness result for $P$-critical metrics, see Corollary~\ref{cor:uniqueness_local}. The definition of the total Chern-Simons class of $E$ is briefly recalled in Section~\ref{sec:Donaldson_functional}.

\smallskip

For each polynomial as in~\eqref{eq:polynomial}, we can define a notion of \emph{stability} on coherent sheaves. First, for a coherent sheaf $\cE$ on a projective variety $X$, let
\begin{equation}\label{eq:pol_bundle}
    P_X(\cE)=\Bigg(\sum_{j=0}^n [\gamma_j]\smile\Chern_j(\cE)\Bigg)\frown[X]\ \in\RR
\end{equation}
where $\Chern_j(\cE)\in H^{2j}(X)$ are the components of the Chern character of $\cE$, $\Chern(\cE)=\sum_j\Chern_j(\cE)$. Note that for a vector bundle $ E\to X$ on a smooth variety and for any~$h\in\Herm(E)$ we have the simple identities
\begin{equation}
    P_X(E)=\int_X\Tr\,P(\curvform(h))=\int_X\Tr\left(\left(\textstyle\sum_{j\geq 0}\gamma_j\right)\wedge\e^{\curvform(h)}\right)
\end{equation}
with the convention that only the top-degree forms are considered in the integral. In particular, $P_X(E)=0$ is a necessary condition for a solution of~\eqref{eq:polynomial_eq} to exist on a bundle $E\to X$: this is simply a normalisation condition on the coefficients $\gamma_j$, that we assume from now on.
\begin{definition}
    Given a polynomial $P(x)$ as in~\eqref{eq:polynomial}, we say that the vector bundle $E$ is \emph{$P$-stable} if $$P_X(S)<0$$ for any coherent saturated subsheaf $S\subset E$ of rank $0<\rk(S)<\rk(E)$.
\end{definition}
This notion of stability (or its \emph{asymptotic} versions; see Section~\ref{sec:asymptotic_regimes} below) encompasses a wide variety of different stability conditions, from Mumford-Takemoto stability to Gieseker stability, all the Bridgeland stability conditions induced by polynomial central charges as in \cite{Bayer}, the notions of \emph{Monge-Ampère stability} introduced by Pingali \cite{Pingali_vbMA}, the high-rank J-stability in Takahashi's work \cite{Takahashi_Jeq_bundles}, and the semistability notions induced by ample classes recently introduced by Toma and collaborators \cite{Toma_semistab_ampleclasses}. The main results of \cite{DMS} can be translated to show that, in an asymptotic regime known as the \emph{large volume limit}, the existence of solutions of~\eqref{eq:polynomial_eq} on a (sufficiently smooth, simple) vector bundle is equivalent to $P$-stability: we will give more details on the results of \cite{DMS}, and their relation to \cite{Toma_semistab_ampleclasses}, in Section~\ref{sec:asymptotic_regimes}.
\begin{remark}
    Note that in contrast with the usual stability notions of Mumford--Takemoto or Gieseker, we restrict ourselves in the definition to considering saturated subsheaves for testing $P$-stability of a bundle. In fact, in \cite{Delloque_adapted}, the first named author provides examples of destabilizing subsheaves with torsion quotients (even in asymp\-totic regimes). The main point is that in general, taking the saturation doesn't increase the value of $P_X$.
\end{remark}
So far, very few results have been obtained linking $P$-stability to the existence of solutions of~\eqref{eq:polynomial_eq} outside asymptotic regimes. Among these, the most notable have been obtained for the J-equation and the \emph{deformed Hermitian Yang-Mills equation} (dHYM, from now on) on a line bundle $E$ over a projective manifold, where the notion of stability becomes empty and the existence of solutions is in fact characterised (under mild hypotheses) by a notion of \emph{positivity} for $E$, reminiscent of the Nakai-Moishezon criterion \cite{Che21,CLT24,DatarPingali_numericalcriterion}. This result was the culmination of a programme started by Lejmi and Székelyhidi \cite{LejmiSzekelyhidi_Jflow} for the J-equation and Collins, Jacob and Yau \cite{CJY20} for the dHYM equation, with the main technical results proved by G.\ Chen in \cite{Che21}. It is remarkable that these results in rank $1$ have been generalised to a large class of $P$-critical equations in rank $1$: see \cite{FangMa_rank1}.

On higher rank bundles, we only have very partial results and a few examples, see~\cite{Cor24}. It is know that in some cases the existence of solutions of~\eqref{eq:polynomial_eq} implies that $E$ is $P$-stable. See \cite[Theorem $1.1$]{KScarpa} for a slightly sharper result. 
\begin{theorem}[\cite{KScarpa}]\label{thm:Pstab}
    Let $E\to X$ be a rank $2$ holomorphic vector bundle on a compact K\"ahler surface, and let $P(x)$ be a degree $2$ polynomial with positive leading order coefficient. If there exists a Hermitian metric $h$ on $E$ that is $P$-critical and \emph{$P$-positive}, then $E$ is $P$-stable.
\end{theorem}
The precise definition of \emph{$P$-positivity} is given in Section~\ref{sec:momentmap}. Briefly, this condition on $h\in\Herm(E)$ is a generalisation of Griffiths positivity guaranteeing that the linearisation of the $P$-critical equation~\eqref{eq:polynomial_eq} is elliptic at $h$.

For the specific choices of polynomial defining the versions of the $J$ equation and the Monge-Ampère equation on vector bundles, Theorem~\ref{thm:Pstab} is due to Takahashi~\cite{Takahashi_Jeq_bundles} and Pingali~\cite{Pingali_vbMA}, respectively. Note that in \cite{KScarpa} the authors consider only those $P$-critical equations that come from a \emph{polynomial central charge} as in \cite{DMS}; the arguments in \cite{KScarpa} however only assume that the equation is defined by a degree~$2$ polynomial, and so they directly apply to our more general case.

Some examples in \cite{KScarpa}, together with the moment map interpretation of the equation, point out that only the existence of solutions of the $P$-critical equation in the set
\begin{equation}
    \HermPos_P(E)\coloneqq \set*{ h\in\Herm(E) \suchthat h\text{ is $P$-positive}}
\end{equation}
can be related to stability properties of $E$. It is natural to wonder if there are some natural conditions that guarantee the existence of $P$-positive metrics; in \cite{KScarpa}, a set of obstructions for a bundle to admit a $P$-positive metric was identified.
\begin{theorem}[\cite{KScarpa}]\label{theo:Ppositive is necessary}
    Assume that the bundle $E\to X$ and a polynomial $P(x)$ are as in Theorem~\ref{thm:Pstab}. If $\HermPos_P(E)\not=\emptyset$, then $P'_Y(Q)>0$ for every irreducible analytic curve $Y\subset X$ and every torsion-free quotient $Q$ of $E_{\restriction Y}$ of rank $\rk(Q)\geq 1$.
\end{theorem}
Here, $P'(x)=\sum_{j=1}^n\gamma_j\frac{1}{(j-1)!}x^{j-1}$ simply denotes the derivative of the polynomial $P(x)=\sum_{j=0}^n\gamma_j\frac{1}{j!}x^{j}$, and note that $Q$ is allowed to coincide with $E_{\restriction Y}$ itself. We refer again to \cite{KScarpa} for a more precise statement of Theorem~\ref{theo:Ppositive is necessary}. In particular, if the leading order coefficient of $P$ is negative rather than positive, the opposite inequality holds for quotients.

The numerical necessary condition in Theorem~\ref{theo:Ppositive is necessary} is called \textit{$P$-positivity of $E$ with respect to quotients}. Leaving the large volume limit case of \cite{DMS} aside, it is a remarkable fact that even on line bundles, the existence of a $P$-critical metric is a highly non-trivial problem (see e.g.\ \cite{Che21}).  The appropriate positivity notion for the holomorphic vector bundle appears then as a key new aspect in the theory of $P$-stability when comparing to more classical stability notions for vector bundles. Despite this observation, there are still few known explicit vector bundles carrying $Z$-critical connections (see e.g.\ \cite{CLSY} for line bundles on elliptic surfaces, \cite{Cor23,Cor24} for homogeneous examples, \cite{JaShe} and \cite{KSD} for line bundles on the blow-up of $\P^n$ in one or two points). In this paper, we develop some devices to check $P$-positivity in practice and to produce new examples of $P$-positive bundles. 

\subsection{\texorpdfstring{$P$}{P}-positivity, \texorpdfstring{$P$}{P}-stability, and equivariance.}
\label{sec:definition Zpositivity}

Assume now that $X$ is a smooth $n$-dimensional projective variety. Fix a polynomial 
$$
 P(x)=\sum_{j=0}^n[\gamma_j]\wedge\frac{1}{j!}x^j\in H^{\bullet}(X,\RR)[x].
$$

\subsubsection{Equivariant $P_\delta$-positivity}

For any coherent sheaf $\cE$ on $X$ and any $d$-di\-men\-sion\-al subvariety $V\subset X$, we  set
\begin{equation}\label{eq:Pstab_restricted}
    P_V(\cE)= \left[\textstyle\sum_{j\geq 0}[\gamma_j] \frown \ch(\cE)\right]^{(d,d)} \smile \cl(V)\in \R.
\end{equation}
where $\cl(V)\in H^{2n-2d}(X,\Q)$ is the cycle class of $V$. A direct adaptation of \cite[Theorem $4.3.13$]{McCarthy_thesis} shows that if $\HermPos_P(E)\not=\emptyset$ then $P_V(\cE)>0$ for any codimension $1$ subvariety $V\subset X$. We expand this condition in the following
\begin{definition}\label{def:Ppositivity}
    Let $\delta=\lbrace \delta_0,\ldots,\delta_{n-1}\rbrace\in \lbrace -1, +1\rbrace^n$ and
    let $\cE$ be a reflexive sheaf on $X$ such that $P_X(\cE)=0$. Then $\cE$ is said to be \emph{$P_\delta$-positive} if for any $1\leq k \leq n-1$, and any subvariety $V\subset X$ with $\dim(V)=k$, one has 
    \begin{equation}\label{eq:inequalitysubvariety}
        \delta_k P_V(\cE)>0.
    \end{equation}
    If in addition the inequality also holds for any $0$-dimensional cycle $V\subset X$, $\cE$ is said to be \emph{strongly $P_\delta$-positive}.
\end{definition}

\begin{remark}
    We'll see in Lemma~\ref{LEM:P positivité forte} that strong $P_\delta$-positivity is equivalent to $P_\delta$-positivity and the fact that $\delta_0\gamma_n>0$. This strong positivity condition then gives a way to take into account the leading coefficient hypothesis of Theorems~\ref{thm:Pstab} and~\ref{theo:Ppositive is necessary}. It ensures in particular that $P$ has maximal degree.
\end{remark}

This definition of positivity extends the notion of $Z$-positivity from \cite{DMS,KScarpa}, where the polynomial $P$ comes from a central charge (see Section~\ref{sec:Zcritic example}). The introduction of the various signs through the $\delta_k$s in Definition~\ref{def:Ppositivity} is motivated by results in \cite{KScarpa,JaShe}, see \cite[Remark $3.15$]{KScarpa} in particular.

The notion of $P_\delta$-positivity for vector bundles should be taken as a first obstruction to the existence of $P$-positive metrics \emph{for suitable polynomials $P$}, that should satisfy an analogue of the supercritical phase condition for the dHYM equation. Indeed, in Section~\ref{sec:equivpositivity}, we provide an example of a $P$-positive metric on a line bundle over the product of three copies of the projective line that is \emph{not} $P_\delta$-positive in the above sense. Note also that in higher rank and on a higher-dimensional base, one should consider a stronger version of positivity to ensure the existence of a positive metric. This stronger version should involve  the signs of the quantities $P_V(Q)$, for all torsion-free quotients $Q$ of $\cE_{\vert V}$, c.f.\ Theorem~\ref{theo:Ppositive is necessary}.

We will be interested in an equivariant version of $P_\delta$-positivity. 
\begin{definition}
 \label{def:equivariant2}
 Let $X$ be a projective $G$-variety, for $G$ an algebraic group. A reflexive sheaf $\cE$ on $X$ is \emph{equivariantly $P_\delta$-positive}  (resp. \emph{equivariantly strongly $P_\delta$-positive}) if Condition~\eqref{eq:inequalitysubvariety} holds for any $G$-invariant subvariety $V\subset X$ with $0<  \dim(V)<\dim(X)$ (resp. with $0\leq  \dim(V)<\dim(X)$). 
\end{definition}
Using Brion's work \cite{Brion}, we easily obtain: 
\begin{theorem}
 \label{theo:intro equiv}
 Let $G$ be a connected and solvable affine algebraic group acting regularly on $X$ a smooth projective variety. Let $\cE$ be a reflexive sheaf over $X$. Then  $\cE$ is (strongly) $P_\delta$-positive if and only if it is $G$-equivariantly (strongly) $P_\delta$-positive.
\end{theorem}

We now give some applications of Theorem~\ref{theo:intro equiv}. For a fixed line bundle $H$ on a compact K\"ahler surface $X$, Khalid and Sj\"ostr\"om Dyrefelt have shown in \cite{KSD} that for any compact region $\cK$ in the set of polynomial central charges (satisfying a minor numerical constraint), there is a finite number of curves $(C_i)_{1\leq i\leq m}$ in $X$ such that for any $Z\in \cK$, $H$ is $Z$-positive if and only if the analogue of condition~\eqref{eq:inequalitysubvariety} is satisfied for each $C_i$ (see also \cite{KSD2} for similar results in higher dimension).
Such finiteness results are crucial in both testing effectively $Z$-positivity and in understanding wall-crossing phenomena under variations of the stability condition. In comparison, we obtain: 
\begin{corollary}
 \label{cor:intro}
 Let $X$ be a smooth projective toric variety and let $\cE$ be a reflexive sheaf over $X$. Then $\cE$ is (strongly) $P_\delta$-positive if and only if for any torus orbit closure $V\subset X$ with $0(\leq)<k=\dim(V)<\dim(X)$, one has 
 $$
\delta_kP_V(\cE)>0.
 $$
\end{corollary}
Note that by the orbit-cone correspondence for toric varieties, the orbit closures are in bijection with the cones in the fan of $X$ \cite[Chapter 3]{CLS}. Hence, they form a finite family only depending on the variety. We thus obtain finiteness of the number of subvarieties to test $P_\delta$-positivity on, \emph{independently of the central charge and of the bundle}. From the definition, $P_\delta$-positivity is an open condition with respect to deformations of $E$. In the toric case, by Corollary~\ref{cor:intro}, it is also an open condition with respect to the choice of $P\in H^{\bullet}(X,\RR)[x]$.

While in Theorem~\ref{theo:intro equiv}, the sheaf is not required to be $G$-equivariant, in the toric case,  torus-equivariant sheaves admit a simple combinatorial description due to Klyachko and Perling \cite{Kl,Perl}. 
Together with the intersection theory of toric orbit closures, which is well-understood (see e.g.\ \cite[Chapter 12]{CLS}), Theorem~\ref{theo:intro equiv} enables to check $P_\delta$-positivity of equivariant sheaves on toric varieties very efficiently.  This is briefly discussed in Section~\ref{sec:toric case}. As  examples, in Section~\ref{sec:toric applications}, we show that the tangent  bundles of Hirzebruch surfaces are positive for the dHYM central charge. We also provide a $3$-dimensional toric Fano example with the same property.

\subsubsection{Equivariant \texorpdfstring{$P$}{P}-stability}
In higher rank, the results from \cite{KScarpa} suggest that relating existence of $P$-critical connections to some stability notion requires to consider $P$-stability. We introduce in Section~\ref{sec:Zstability} \emph{equivariant $P$-stability} for toric sheaves, and explain the difficulties to overcome in order to show that it implies its non-equivariant version. For slope (or Gieseker) stability, it is known that stability of a toric sheaf is implied by its equivariant stability \cite{Koo}, and one of the missing features to extend the arguments  here is the existence of Jordan--H\"older and Harder--Narasimhan filtrations. Ultimately, the issue comes from the fact that the process of saturating a sheaf doesn't increase what we could call its $P$-slope in general, which is a source of counter-examples, as explained in \cite{Delloque_adapted}.  Nevertheless, equivariant $P$-stability is expected to be a necessary condition for solving the $P$-critical equation, and can be checked efficiently for toric sheaves. In the large volume limit case, for suitable polynomials (called adapted to torsion-free sheaves in \cite{Delloque_adapted}, see Section~\ref{sec:Zstability}), we obtain: 

\begin{theorem}
 \label{theo:intro equiv Zstab asymptotic}
 Let $\cE$ be an equivariant reflexive sheaf over a $T$-variety $X$, and let $P_r$ be an asymptotic polynomial equation adapted to torsion-free sheaves. Then $\cE$ is asymptotically $P_r$-stable if and only if it is equivariantly asymptotically $P_r$-stable.
\end{theorem}

As a corollary, relying on \cite{DMS}, we obtain in Section~\ref{sec:application asymptotic case} an example of a rank $3$-bundle on a $3$-dimensional manifold that is not slope polystable, but still carries $P$-critical connections. This example provides solutions to the gauge theoretical equations associated to the recent $\alpha$-semistability conditions introduced in \cite{Toma_semistab_ampleclasses}.

\subsection{Blow-ups along points}

Leaving the equivariant context, in another direction, we use blow-ups to construct $P_\delta$-positive bundles by pullbacks. Let $\pi: \tilde X \to X$  be the blow-up of a smooth projective variety of dimension $n\geq 2$ along a point and we call $D$ the exceptional divisor. Consider the polynomial
$$
P_V(\cE) = \left(\sum_{j = 0}^n [\gamma_j] \smile \Chern(\cE)\right) \frown [V],
$$
which naturally induces a polynomial on $\tilde{X}$,
$$
\tilde{P}_{0,V}(\cE) = \left(\sum_{j = 0}^n \pi^*[\gamma_j] \smile \Chern(\cE)\right) \frown [V].
$$
With this definition, it is easy to see that the pullback of any bundle $\cE$ on $X$ is not $\tilde{P}_{0,\delta}$-positive because for all $V \subset D$ sub-variety of positive dimension, $\tilde{P}_{0,V}(\pi^*\cE) = 0$. To avoid it and be able to construct $P_\delta$-positive bundles by pullback, we deform the coefficients of $\tilde{P}$,
$$
\tilde{P}_{\varepsilon,V}(\cE) = \left(\sum_{j = 0}^n (\pi^*[\gamma_j] + \varepsilon_j\cl(D)^{n - j}) \smile \Chern(\cE)\right) \frown [V].
$$
where each $\varepsilon_j$ is a small real number. When $\cE$ is a $P_\delta$-positive bundle, we look for necessary and sufficient conditions of the $\varepsilon_j$ for $\pi^*\cE$ to be $\tilde{P}_{\varepsilon,\delta}$-positive. We this, we need to introduce a uniform version of $P_\delta$-positivity.

Notice that $P_V(\cE)$ can be expressed as
$$
P_V(\cE) = \left(\sum_{j = 0}^n [\gamma_j] \smile \Chern(\cE) \smile \cl(V)\right) \frown [X].
$$
We have $\cl(V) \in H^{n - k,n - k}(X,\RR)$ where $k = \codim(V)$. We define the effective cone $\Eff_k(X)\subset H^{n - k,n - k}(X,\RR)$ by
$$
\Eff_k(X) = \left\{\sum_{i \in I} a_i\cl(V_i)\middle|I \textrm{ finite}, V_i \subset X \textrm{ of dimension } k, a_i \geq 0 \right\} ,
$$
and the pseudo-effective cone as its closure. As the effective cone has a non-empty interior in $H^{n - k,n - k}(X,\RR)$, we may extend $P_\bullet(\cE)$ to the whole $H^{n - k,n - k}(X,\RR)$ by linearity, for all $k$. Clearly, $\cE$ is (strongly) $P_\delta$-positive if and only if,
$$
P_X(\cE) = 0, \qquad \forall\, 0 <(\leq)\, k \leq n - 1,\, \forall\, c \in \Eff_k(X), \quad \delta_kP_c(\cE) > 0.
$$

\begin{definition}
    We say that $\cE$ is \textit{(strongly) uniformly $P_\delta$-positive} if
    $$
    P_X(\cE) = 0, \qquad \forall\, 0 <(\leq)\, k \leq n - 1,\, \forall\, c \in \overline{\Eff_k(X)}, \quad \delta_kP_c(\cE) > 0.
    $$
\end{definition}

Our main result on the $P_\delta$-positivity of pullbacks along blow-ups is the following.

\begin{theorem}\label{THE:P positivité pull-back}
    Let $\cE$ be a (strongly) uniformly $P_\delta$-positive vector bundle on $X$. There is a constant $\eta > 0$ only depending on $X$, the blown-up point, the $[\gamma_j]$ and $\cE$ such that for all $\varepsilon = (\varepsilon_0,\ldots,\varepsilon_n)$ with for all $k$, $|\varepsilon_k| < \eta$, we have equivalence between
    \begin{enumerate}
        \item $\pi^*\cE$ is (strongly) $\tilde{P}_{\varepsilon,\delta}$-positive.
        \item $\pi^*\cE$ is (strongly) uniformly $\tilde{P}_{\varepsilon,\delta}$-positive.
        \item $\varepsilon_0 = (-1)^n\frac{\Chern_n(\cE) \frown [X]}{\rk(\cE)}\varepsilon_n$ and for all $1 \leq k \leq n - 1$, $(-1)^k\delta_k\varepsilon_{n - k} > 0$.
    \end{enumerate}
\end{theorem}

The equality $\varepsilon_0 = (-1)^n\frac{\Chern_n(\cE) \frown [X]}{\rk(\cE)}\varepsilon_n$ comes from the condition $\tilde{P}_{\varepsilon,\tilde{X}}(\pi^*\cE) = 0$. The uniform $\tilde{P}_{\varepsilon,\delta}$-positivity of $\pi^*\cE$ with respect to sub-varieties of $X$ not included in the exceptional divisor is ensured by the uniform $P_\delta$-positivity of $\cE$ as long as the $|\varepsilon_k|$ are small enough. And the inequalities $(-1)^k\delta_k\varepsilon_{n - k} > 0$ are equivalent to $\tilde{P}_{\varepsilon,\delta}$-positivity of $\pi^*\cE$ with respect to the sub-varieties of $D$.

From Corollary~\ref{cor:intro}, on toric varieties, $P_\delta$-positivity is equivalent to its uniform version. Hence, we may apply Theorem~\ref{THE:P positivité pull-back} to the examples of $P_\delta$-positive bundles constructed in Section~\ref{sec:toric applications} to produce new examples of $P_\delta$-positive bundles of rank strictly greater than $1$ on iterated toric blow-ups of either Hirzebruch surfaces or the Fano $3$-fold considered in Proposition~\ref{prop:introFano3fold}.

 \subsection*{Organisation of the paper} In Section~\ref{sec:polynomial curv eq}, we provide the moment map interpretation of $P$-critical metrics and the associated Lagrangian. We study its second variation and obtain a local uniqueness result. Finally, we give some specific examples of $P$-critical equations, and introduce their asymptotic versions. In Section~\ref{sec:equivpos}, we give the proof of Theorem~\ref{theo:intro equiv} and discuss its applications for toric sheaves. We also provide an example of $P$-critical metric on a line bundle that is not $P_\delta$-positive. We then introduce equivariant $P$-stability and prove Theorem~\ref{theo:intro equiv Zstab asymptotic}.  Section~\ref{sec:blowups} is devoted to blow-ups, and the proof of Theorem~\ref{THE:P positivité pull-back}. In Section~\ref{sec:applications}, we derive  some applications, including an example of a Fano $3$-fold of Picard rank $3$ whose tangent bundle is $Z$-positive for the dHYM central charge.
 
\subsection*{Acknowledgements}  
The authors would like to thank Michel Brion for kindly answering their questions on rational equivalence over toric varieties and for pointing to them references \cite{Brion,FuMacStu}. They are also grateful to Ruadha\'i Dervan and Lars Martin Sektnan  for many stimulating exchanges on the topic.
R.D. and C.T. are partially supported by Centre Henri Lebesgue ANR-11-LABX-0020-01, and the grants MARGE ANR-21-CE40-0011 and BRIDGES ANR--FAPESP ANR-21-CE40-0017.
A.N. is supported by the FAPESP post-doctoral grant number 2023/06502-0.
C.S. is supported by a MSCA Postdoctoral Fellowship {\euflag} under grant agreement No 101149320.

\section{Polynomial curvature equations}
\label{sec:polynomial curv eq}

\subsection{A moment map picture}\label{sec:momentmap}

An unpleasant feature of~\eqref{eq:polynomial_eq} is that it might fail to be elliptic. To understand the situations in which the equation is elliptic at least near a solution, consider the linearisation of the operator $h \mapsto\e^{\curvform(h)}$, defined on the space $\Herm(E)$ of Hermitian metrics on $E$.

Taking a one-parameter path of metrics $h_t=\e^{tV}h$, we have
\begin{equation}
    \partial_t\curvform(h_t)=\frac{\I}{2\pi}\left(\nabla^{0,1}\nabla^{1,0}-\nabla^{1,0}\nabla^{0,1}\right)V
\end{equation}
and our polynomial operator linearises as
\begin{equation}
    \partial_{t=0}P(\curvform(h_t))=\sum_{j=1}^n\gamma_j\wedge\frac{1}{j!}\sum_{1\leq p\leq j}\curvform(h)^{p-1}\wedge(\partial_{t=0}\curvform(h_t))\wedge\curvform(h)^{j-p}.
\end{equation}
We can obtain a more compact expression by considering, as in \cite{DMS}, the multi-linear product on $\Alt^\bullet(\End E)$ defined by
\begin{equation}
    \left[A_1\wedge\dots\wedge A_j\right]_{\sym}\coloneqq \frac{1}{j!}\sum_{\sigma\in S_j}(-1)^{\mathrm{gr\,sgn}(\sigma)}A_{\sigma(1)}\wedge\dots\wedge A_{\sigma(j)}
\end{equation}
where the graded sign of a permutation, $\mathrm{gr\,sgn}(\sigma)$, is sign obtained by permuting $A_1\wedge\dots A_j$ to $A_{\sigma(1)}\wedge\dots\wedge A_{\sigma(j)}$ as differential forms (so, ignoring $\End E$ factors). By a slight abuse of notation then the linearisation of $P(\curvform(h))$ becomes
\begin{equation}
    \partial_{t=0}P(\curvform(h_t))=\sum_{j=1}^n\gamma_j\wedge\frac{1}{(j-1)!}\big[\curvform(h_0)^{j-1}\wedge\partial_{t=0}\curvform(h_t)\big]_{\sym}.
\end{equation}
As all the $\gamma_j$ $\wedge$-commute with endomorphism-valued forms, the linearisation can be expressed simply as
\begin{equation}\label{eq:lin_first}
    \partial_{t=0}P(\curvform(h_t))=[P'(\curvform(h))\wedge\partial_{t=0}\curvform(h_t)]_{\sym}.
\end{equation}

\begin{definition}[\cite{DMS}]
    We say that a connection $\nabla$ on $E$ is \emph{$P$-positive} if the $2(n-1)$ $\End E$-valued form $P'(\curvform(\nabla))$ is positive definite, i.e.\ for any $p\in X$ and any non-zero $\xi\in T^{0,1}_p{}^*X\times\End E_p)$
    \begin{equation}
        \Tr\I\left[P'(\curvform(\nabla))\wedge \xi^*\wedge \xi\right]_{\sym}>0.
    \end{equation}
    A Hermitian metric is said to be $P$-positive if its associated Chern connection is, and the set of all $P$-positive Hermitian metrics on $E$ will be denoted by $\HermPos_P(E)$.
\end{definition}
Reasoning as in \cite[Lemma $2.36$]{DMS} one readily obtains
\begin{lemma}
	If $\nabla$ is a $P$-positive connection on $E$, then the $P$-critical equation is elliptic at $\nabla$.
\end{lemma}
This positivity condition is also crucial for the moment map interpretation of the $P$-critical equation, as remarked in \cite{DMS}.
\begin{proposition}[\cite{DMS}]\label{prop:momentmap}
	Fix some $h_0\in\Herm(E)$, and let $\Holstr$ be the set of $h_0$-unitary connections on $E$. The action of the $h_0$-unitary gauge group on $\Holstr$ is Hamiltonian with respect to the Hermitian pairing
	\begin{equation}\label{eq:momentmap_pairing}
	    \langle a,b\rangle=-\I\int_X\Tr\left[P'(\curvform(\nabla))\wedge a\wedge b^*\right]_{\sym},
	\end{equation}
	and a moment map is given by $\nabla\mapsto P(\curvform(\nabla))$.
\end{proposition}
Note that the pairing~\eqref{eq:momentmap_pairing} is positive precisely when $\nabla$ is $P$-positive. This moment map description of the $P$-critical metrics is a straightforward adaptation of \cite[Theorem $2.45$]{DMS}.

\subsection{A Lagrangian for $P$-critical metrics}\label{sec:Donaldson_functional}

Given a polynomial $P(x)$ as in~\eqref{eq:polynomial}, we now describe a modified version of Donaldson's Lagrangian for the Hermite-Einstein equation, whose Euler-Lagrange equation is the $P$-critical equation~\eqref{eq:polynomial_eq}.

\smallbreak

We start by collecting a few basic facts about the Bott-Chern representatives of secondary classes, also known as Chern-Simons classes, from \cite[\S$1.2$]{Donaldson85}.

\begin{proposition}\label{prop:ChernSimons}
    Consider a holomorphic vector bundle $E\to X$ of rank $r$, an integer $0\leq p\leq r$, and a Hermitian metric $h_0$ on $E$. Then, for any $p$-linear, totally symmetric, and $\mrm{GL}(r,\mathbb{C})$-invariant function $\varphi_p:\mathfrak{gl}(r,\mathbb{C})^p\to\mathbb{C}$, there is a functional
    \begin{equation}
    \begin{split}
        R:\Herm(E)&\to\frac{\Alt^{p-1,p-1}(X)}{\range\del+\range\delbar}\\
        h&\mapsto R(h,h_0)
    \end{split}
    \end{equation}
    with the following properties:
    \begin{enumerate}
        \item $R(h_0,h_0)=0$ and $R(h_2,h_0)=R(h_2,h_1)+R(h_1,h_0)$;
        \item $\partial_tR(h_t,h_0)=- p\,\varphi_p(h_t^{-1}\dot{h}_t,\curvform(h_t),\dots,\curvform(h_t))$;
        \item $\I\del\delbar R(h_1,h_0)=\varphi_p(\curvform(h_1),\dots,\curvform(h_1))-\varphi_p(\curvform(h_0),\dots,\curvform(h_0))\in\Alt^{p,p}(X)$.
    \end{enumerate}
\end{proposition}

We will be mostly interested with a representative $\BC(h,h_0)$ of the \emph{Chern-Simons character}, that satisfies
\begin{equation}
    2\pi\I\del\delbar\BC(h,h_0)=\Tr\big(\e^{\curvform(h)}-\e^{\curvform(h_0)}\big).
\end{equation}
We can define $\BC(h,h_0)$ as follows: for any $j\geq 0$, consider the symmetric multilinear functions on $\mathfrak{gl}(r,\mathbb{C})$ defined by
\begin{equation}
    \varphi_j(A_1,\dots,A_j)=\frac{1}{2\pi}\Tr\left[A_1 \wedge \dots \wedge A_j\right]_{\sym}.
\end{equation}
Consider the Chern-Simons classes corresponding to these $\varphi_j$, and their representatives $R_j(h,h_0)$. The total Chern-Simons class is represented by
\begin{equation}
    \BC(h,h_0)\coloneqq\textstyle\sum\nolimits_{j\geq 0}\frac{1}{j!}R_j(h,h_0)
\end{equation}
and taking the derivative along a path of metrics $h_t$ one obtains
\begin{equation}
    \partial_t\BC(h_t,h_0)=
    -\frac{1}{2\pi}\Tr\left(h_t^{-1}\dot{h}_t\cdot\e^{\curvform(h_t)}\right).
\end{equation}
Given a reference metric $h_0$ and a polynomial $P(x)=\sum\gamma_j\wedge\frac{1}{j!}x^j$, we consider the $P$-critical Lagrangian defined in Proposition~\ref{prop:Lagrangian}:
\begin{equation}
    \PLagr(h,h_0)=2\pi\int_X{\textstyle\sum_{j\geq 0}}\gamma_j\wedge\frac{1}{j!}R_j(h,h_0)=2\pi\int_X\left(\textstyle\sum_{j\geq 0}\gamma_j\right)\wedge\BC(h,h_0).
\end{equation}
Then, the differential of $\PLagr(h)$ is
\begin{equation}
\begin{split}
    \partial_{t=0}\PLagr(h+t\dot{h},h_0)=&2\pi\int_X\left(\textstyle\sum_{j\geq 0}\gamma_j\right)\wedge\partial_{t=0}\BC(h+t\dot{h},h_0)=\\
    =&-\int_X\Tr\left(h^{-1}\dot{h} (\textstyle\sum_{j\geq 0}\gamma_j) \e^{\curvform(h)}\right)=-\int_X\Tr\left(h^{-1}\dot{h}\,P(\curvform(h))\right).
\end{split}
\end{equation}
This computation shows that the Euler-Lagrange equation of $\PLagr$ is the $P$-critical equation~\eqref{eq:polynomial_eq}, proving Proposition~\ref{prop:Lagrangian}. Recalling the moment map description of the $P$-critical equation we have the following, along the lines of \cite[\S$4$]{Donaldson85}.
\begin{proposition}
    Up to the possible addition of a constant, $\PLagr$ is the Kempf-Ness functional for the Hamiltonian action of Proposition~\ref{prop:momentmap}.
\end{proposition}
In particular, we expect $\PLagr$ to be a \emph{convex} functional, in some sense, at least on the space $\HermPos_P(E)$ of $P$-positive metrics on $E$. 

To confirm this expectation, let us start by computing the second variation of the $\PLagr$ functional. We can start from the first variation of $\PLagr$ to obtain
\begin{equation}
\begin{split}
    \partial^2_t\PLagr(h_t)=&\int_X\Tr\left(-h_t^{-1} \dot{h}_t h_t^{-1} \dot{h}_t\,P(h_t)\right) +\int_X\Tr\left(h_t^{-1}\ddot{h}_t\,P(\curvform(h_t))\right)\\
    &+\int_X\Tr\left(h_t^{-1}\dot{h}_t\,\left[P'(\curvform(h))\wedge\frac{\I}{2\pi}\left(\nabla^{0,1}\nabla^{1,0}-\nabla^{1,0}\nabla^{0,1}\right)(h^{-1}_t\dot{h}_t)\right]_{\sym}\right).
\end{split}
\end{equation}
As the trace is invariant under cyclic permutations, we integrate by parts and get
\begin{equation}\label{eq:secondvar}
\begin{split}
\partial^2_t\PLagr(h_t)=&\int_X\Tr\left(h_t^{-1}(\ddot{h}_t-\dot{h}_th_t^{-1}\dot{h}_t) P(\curvform(h_t))\right)\\
&-\frac{\I}{2\pi}\int_X\Tr\left[P'(\curvform(h_t))\wedge\left(\nabla_t^{1,0}(\dot{h}_th_t^{-1})\right)^*\wedge\nabla_t^{1,0}(\dot{h}_th_t^{-1})\right]_{\sym}\\
=&\int_X\Tr\left(h_t^{-1}(\ddot{h}_t-\dot{h}_th_t^{-1}\dot{h}_t)\,P(\curvform(h_t))\right)
+\frac{1}{2\pi}\norm*{\nabla_t^{1,0}(\dot{h}_th_t^{-1})}^2_{P,h_t}
\end{split}
\end{equation}
where $\norm{\cdot}_{P,h}$ denotes the pairing in Proposition~\ref{prop:momentmap}, which is a norm provided that $h_t$ is $P$-positive. Note that~\eqref{eq:secondvar} shows that $\PLagr$ is convex around any $h\in\HermPos_P(E)$ satisfying $P(\curvform(h))=0$, i.e.\ around $P$-critical and $P$-positive Hermitian metrics.

Given any $h\in\Herm(E)$, the other Hermitian metrics on $E$ can be parametrised as $h'=\e^{A}.h\coloneqq h(\e^{A}\cdot,\e^{A}\cdot)$, for an $h$-Hermitian endomorphism $A$. If we consider a path $h_t\in\Herm(E)$ defined as $h_t=\e^{t A}.h$, we obtain for the endomorphisms $\dot{h}_t$ and $\ddot{h}_t$ the expressions
\begin{equation}
\begin{split}
\dot{h}_t=&h(\e^{tA}A\cdot,\e^{tA}\cdot)+h(\e^{tA}\cdot,\e^{tA}A\cdot)=2\,h_t(A\cdot,\cdot)\\
\ddot{h}_t=&2\,\dot{h}_t(A\cdot,\cdot)=4\,h_t(A^2\cdot,\cdot)
\end{split}
\end{equation}
so we get $\dot{h}_th_t^{-1}=2\,A$ and $\ddot{h}_th_t^{-1}=4\,A^2=(\dot{h}_th_t^{-1})^2$. Putting this together with our computation for the second variation of $\PLagr$, we conclude
\begin{lemma}\label{lemma:exponentialpaths}
Along a path of Hermitian metrics $h_t=\e^{t A}.h=h(\e^{t A}\cdot,\e^{t A}\cdot)$ for $A$ an $h$-unitary endomorphism of $E$, one has
\begin{equation}
\partial^2_t\PLagr(h_t)=\frac{2}{\pi}\norm*{\nabla_t^{1,0}A}^2_{P,h_t}
\end{equation}
with respect to the pairing~\eqref{eq:momentmap_pairing}. In particular, if $\e^{t A}.h$ is $P$-positive for all $t\geq 0$, then $\PLagr$ is convex along this path, and it is strictly convex unless $A$ is holomorphic.
\end{lemma}
It is not a priori clear if these paths lie in $\HermPos_P(E)$, even if the endpoints are $P$-positive. To establish a general uniqueness result for $P$-critical metrics it will likely be necessary to consider different paths of metrics. As shown by recent examples in \cite{BallalPingali} however, $P$-positivity might not be preserved even along a continuity path. In any case, as $P$-positivity is a local property, we deduce from Lemma~\ref{lemma:exponentialpaths} a local uniqueness result for the solutions of the $P$-critical equation.
\begin{corollary}\label{cor:uniqueness_local}
    Assume that $h\in\HermPos_P(E)$ is a $P$-critical metric. If $h'\in\Herm(E)$ is another $P$-critical metric that is sufficiently close to $h$, then $h'$ and $h$ are conjugated by a holomorphic automorphism of $E$. 
\end{corollary}

\subsection{Asymptotic regimes and some examples}\label{sec:asymptotic_regimes}

\subsubsection{Z-critical equations}
\label{sec:Zcritic example}
We are mostly interested in those polynomial equations and polynomial stability conditions that fit into the framework of \emph{$\Zch$-critical equations} introduced in \cite{DMS}.
\begin{definition}\label{def:centralcharge}
    A \emph{polynomial central charge} on $X$ is given by a triple $\Zch=(L,\rho,U)$ where $L$ is a polarisation on $X$, $\rho$ is a \emph{stability vector} and $U$ is a \emph{unipotent class}:
    \begin{itemize}
        \item $\rho=(\rho_0,\ldots,\rho_n)\in(\C^*)^{n+1}$ such that $\Im(\rho_j/\rho_{j+1})>0$ for all $j$,
        \item $U=1 + U_1+\ldots +U_n\in H^\bullet(X,\R)$ such that $U_j\in H^{j,j}(X,\R)$.
    \end{itemize}
    Given such a polynomial central charge $\Zch=(L,\rho,U)$, we define a function $Z_X$ on the Grothendieck ring of $X$ by the formula
    \begin{equation}\label{eq:polynomialchargebundles}
        \Zch_X(\cE)=\int_X \sum_{j=0}^n \rho_j c_1(L)^j\wedge U\wedge \ch(\cE).
    \end{equation}
\end{definition}
Following \cite{DMS}, given a polynomial central charge $\Zch$, we consider the following stability condition.
\begin{definition}\label{def:Zpositive}
    Let $\cE$ be a reflexive sheaf on $X$. Then $\cE$ is said to be \emph{$\Zch$-stable} (resp. $\Zch$-semistable) if for any saturated reflexive subsheaf $\cF\subset \cE$ with $0<\rank(\cF)<\rank(\cE)$ one has
    \begin{equation}\label{eq:ineqalitysheaves}
        \Im\left(\Zch_X(\cF)\overline{\Zch_X(\cE)}\right)<0\: (\text{resp.} \leq 0).
    \end{equation}
\end{definition}
This stability condition, for a vector bundle $E\to X$, is related to the existence of \emph{$\Zch$-critical metrics}, \cite{DMS,KScarpa}. More precisely, given $\Zch=(L,\rho,U)$, fix a K\"ahler representative $\omega$ of $\chern_1(L)$ and a unitary representative $u=1+u_1+\dots+u_n$ of the unitary class $U$. Then, a metric $h\in\Herm(E)$ is \emph{$\Zch$-critical} if
\begin{equation}\label{eq:Zcrit}
    \Im\left[\overline{\Zch_X(E)}\left(\textstyle\sum_{j\geq 0}\rho_j\omega^j\right) \wedge u \wedge\e^{\curvform(h)}\right]^{\text{deg. }2n\text{ part}}=0.
\end{equation}
Of course this can be regarded as an equation of the type~\eqref{eq:polynomial_eq} for a polynomial $P(x)=\sum\gamma_j\frac{1}{j!}x^j$ with coefficients
\begin{equation}\label{eq:polynomialcoeffs_Z}
    \gamma_j=\sum_{i+k=j}\Im\left(\overline{\Zch_X(E)}\rho_k\right)\omega^k\wedge u_i,
\end{equation}
with the convention that $u_0=1$. Note that the resulting polynomial, and the corresponding stability condition, depend on the topology of the bundle $E$. Thus, we denote by $P_{(Z,E)}$ the polynomial defined by the coefficients~\eqref{eq:polynomialcoeffs_Z}.
\begin{proposition}
    Let $E\to X$ be a holomorphic vector bundle, and let $\Zch$ be a polynomial central charge. Then, $E$ is $\Zch$-stable if and only if it is $P_{(Z,E)}$-stable.
\end{proposition}

The main result of \cite{DMS} is that, for simple and sufficiently smooth bundles, $\Zch$-stability characterises the existence of $\Zch$-critical metrics in an asymptotic regime known as the \emph{large volume limit}. This means that, given a central charge $\Zch=(L,\rho,U)$, one considers the stability condition and the differential equation induced by the rescaled central charge $\Zch_r=(L^r,\rho,U)$ for very large values of $r$. Under this rescalings, the coefficients $\gamma_j(r)$ of~\eqref{eq:polynomialcoeffs_Z} scale as
\begin{equation}
    \gamma_j(r)=\Im(\overline{\rho}_n\rho_{n-j})\,r^{2n-j}\,L^n\,\rk(E)\,\omega^{n-j}+\text{lower order terms},
\end{equation}
so that the equation defined by the polynomial $P_r(x)\coloneqq \sum\gamma_j(r) \frac{1}{j!} x^j$ is simply a weak Hermite-Einstein equation, to a leading order in $r$.

\subsubsection{Asymptotic stability conditions}
\label{sec:asymptotic Toma et al}
A particular case of $P$-stability has already been considered in \cite{Toma_semistab_ampleclasses}. Given a coherent sheaf $\cE$ on $X$ and \emph{ample} cycle classes $\alpha=(\alpha_0,\dots,\alpha_n)$ such that $\alpha_j\in H_{2(n-j)}(X)$, the $\alpha$-Hilbert polynomial of $E$ is defined as
\begin{equation}\label{eq:Pcondition-Toma}
    P_\alpha(\cE,m)=\sum\nolimits_{j\geq 0}\frac{1}{j!}\int_X\Chern(\cE)\alpha_j\Todd_X\,m^j.
\end{equation}
A (semi)stability condition is then defined by comparing $P_\alpha(\cE,m)$ and $P_\alpha(\cF,m)$, for $\cF\subset\cE$, for large values of $m$. This is clearly a particular case of an (asymptotic) polynomial stability condition; in fact, it is shown in \cite[Remark $2.8$]{Toma_semistab_ampleclasses} that the stability conditions defined by ample classes is a particular case of the ``Bridgeland-like'' stability conditions defined by Bayer \cite{Bayer}, albeit in a slightly more general setting than the one considered in \cite{DMS}. 

These \emph{asymptotic} stability conditions include in particular Gieseker stability; see also \cite[\S$4.2$]{KScarpa} for an approach casting Gieseker stability in terms of asymptotic $\Zch$-stability for a certain choice of central charge.

\smallbreak

We can slightly generalise the asymptotic stability notions defined by polynomial central charges \cite{DMS}, or ample classes \cite{Toma_semistab_ampleclasses}, to general polynomial stability conditions.
\begin{definition}\label{def:asymptotic_eqs}
    An \emph{asymptotic} polynomial equation for a Hermitian metric on a vector bundle $E\to X$ is a family of polynomial equations $\set*{ P_r(\curvform(h))=0 \suchthat r\in\R }$, defined by $n+1$ functions of one real variable
    \begin{equation}
        r\mapsto\gamma_j(r)\in\Alt^{n-j,n-j}(X,\R)\cap\ker\dd
    \end{equation}
    that satisfy $\bar{\gamma}_1\not\equiv 0$ and the conditions
    \begin{enumerate}
        \item $\lim_{r\to+\infty}r^{-(n-j)}\gamma_j(r)=\bar{\gamma}_j$ for every $j=1,\dots,n$;
        \item $\lim_{r\to+\infty}r^{-(n-1)}\gamma_0(r)=\bar{\gamma}_0$,
    \end{enumerate}
    where $(\bar{\gamma}_0,\dots,\bar{\gamma}_n)$ is an $(n+1)$-uple of closed differential forms on $X$, and the convergence in $(1)$ and $(2)$ is with respect to the $\m{C}^0$-norm defined by a (any) fixed Riemannian metric on $X$.
\end{definition}
These hypotheses guarantee that, at the highest order in $r$, any asymptotic polynomial equation coincides with a (weak) Hermite-Einstein equation
\begin{equation}
    \bar{\gamma}_0\otimes\Id_E+\curvform(h)\wedge\bar{\gamma}_1=0.
\end{equation}
If $\bar{\gamma}_1$ is a \emph{positive} closed $(n-1,n-1)$-form on $X$, it is possible to relate the Mumford-Takemoto (semi)stability of a vector bundle with respect to the class $[\bar{\gamma}_1]$ with both the $P_r$-stability and the existence of $P_r$-critical metrics for $r\gg 0$. The following shows that it is natural to only consider those asymptotic equations for which $\bar{\gamma}_1>0$.
\begin{lemma}
    Consider an asymptotic polynomial equation defined by the polynomials $P_r(x)$. With the notation of Definition~\ref{def:asymptotic_eqs}, $\bar{\gamma}_1$ is a positive $(n-1,n-1)$-form if and only if any connection on $E$ if $P_r$-positive for $r\gg 0$.
\end{lemma}
\begin{proof}
    We follow the proof of \cite[Lemma $3.8$]{DMS}. Given any connection $\nabla$ on $E$ and any $\xi\in\Alt^{0,1}(\End E)$, we have
    \begin{equation}
        \Tr\I\left[P'_r(\curvform(\nabla))\wedge \xi^*\wedge \xi\right]_{\sym} = r^{n-1}\,\bar{\gamma}_1 \wedge \I\Tr(\xi^*\wedge\xi) + O(r^{n-2}).
    \end{equation}
    As $\bar{\gamma}_1\not=0$ by assumption, $\nabla$ is $P_r$-positive for any $r\gg0$ if and only if $\bar{\gamma}_1 \wedge \I\Tr(\xi^*\wedge\xi)>0$ for any non-zero $\xi$. Choosing $\xi$ of the form $\alpha\otimes\Id_E$ for a $(0,1)$-form $\alpha$ on $X$, we see that $\bar{\gamma}_1$ is positive. The converse is easy to check by choosing a positive $(1,1)$-form $\omega$ such that $\omega^{n-1}=\bar{\gamma}_1$ and a local coordinate system in which $\omega$ is represented by a diagonal matrix.
\end{proof}
The powerful techniques of \cite[\S$4$]{DMS} can be adapted to a wide variety of infinite-dimensional moment map equations, and in particular to show a link between the existence of $P$-critical metrics and $P$-stability of a bundle in asymptotic regimes. See also \cite{Delloque} for similar considerations on slope-stability of bundles on \emph{balanced} (rather than K\"ahler, as in \cite{DMS}) manifolds.
\begin{theorem}[\cite{DMS}]\label{thm:DMS_polynomial}
    Let $E\to X$ be a simple vector bundle, and consider the asymptotic polynomial equation $P_r(\curvform(h))=0$ defined by $(\gamma_0,\dots,\gamma_n)$.
    Assume that $\bar{\gamma}_1>0$, and that $E$ is sufficiently smooth with respect to $\bar{\gamma}_1$. Then, the following are equivalent:
    \begin{enumerate}
        \item $E$ is $P_r$-stable with respect to subbundles for $r\gg 0$;
        \item there is a constant $C>0$ such that for every $r\gg 0$ there exists $h_r\in\Herm(E)$ satisfying $P_r(\curvform(h_r))=0$ and $\norm{h_k}_{\m{C}^2}<C$.
    \end{enumerate} 
\end{theorem}
This is mostly a restatement of \cite[Theorem $1.1$]{DMS} in the context of polynomial stability conditions, but a few comments are in order. First, the class $[\bar{\gamma}_1]\in H^{n-1,n-1}(X)$ defines a notion of slope-polystability, and each condition $(1)$ and $(2)$ of Theorem~\ref{thm:DMS_polynomial} implies that $E$ is $[\bar{\gamma}_1]$-semistable (c.f.\ \cite[Lemma $2.11$ and Lemma $2.28$]{DMS}). As in \cite{Delloque} we can then consider a Jordan-H\"older filtration $0=S_0\subset\dots\subset S_N\subset E$ induced by $[\bar{\gamma}_1]$ and the graded object $Gr_{[\bar{\gamma}_1]}(E)=\bigoplus S_j/S_{j-1}$. The hypothesis that $E$ is sufficiently smooth with respect to $\bar{\gamma}_1$ simply asks that $Gr_{[\bar{\gamma}_1]}(E)$ be a vector bundle, rather than just a coherent sheaf.

\section{Equivariant positivity and stability}
\label{sec:equivpos}

In this section, we fix a smooth projective variety $X$ and a polynomial 
$$
 P(x)=\sum_{j=0}^n[\gamma_j]\wedge\frac{1}{j!}x^j\in H^{\bullet}(X,\RR)[x].
$$
\subsection{Equivariant $P$-positivity}
\label{sec:equivpositivity}
We start with the proof of Theorem~\ref{theo:intro equiv}.

\begin{proof}[Proof of Theorem~\ref{theo:intro equiv}]
Under the assumptions of Theorem~\ref{theo:intro equiv}, effective cycles are rationally equivalent to $G$-invariant effective cycles on $X$, c.f.\ \cite[Theorem 1.3]{Brion} (see also \cite[proof of Theorem 1]{FuMacStu}). The proof is then straightforward: for any $d$-dimensional subvariety $V\subset X$, there are positive integers $(a_i)_{1\leq i \leq r}$ and $G$-invariant cycles $(V_i)_{1\leq i\leq r}$ such that 
$$
V\sim \sum_{i=1}^r a_i V_i.
$$
Then, the cycle class map gives
$$
\cl(V)= \sum_{i=1}^r a_i\, \cl(V_i),
$$
hence by equation~\eqref{eq:Pstab_restricted}
$$
P_{V}(\cE)=\sum_{i=1}^r a_i\, P_{V_i}(\cE).
$$
Thus, $P_\delta$-positivity (that is Condition~\eqref{eq:inequalitysubvariety}) with respect to $G$-invariant cycles implies $P_\delta$-positivity with respect to all subvarieties on $X$.
\end{proof}

\begin{remark}
 A similar result was known for the $J$-equation on toric manifolds. Indeed, in \cite{ColSze}, it is proved directly that the relevant equivariant version of positivity implies the existence of solutions to the $J$-equation.
\end{remark}

Note that, following Section~\ref{sec:Zcritic example}, Theorem~\ref{theo:intro equiv} applies to the case of the positivity notion associated to polynomial central charges, that we now recall. Given a polynomial central charge $Z=(L,\rho,U)$, for any coherent sheaf $\cE$ on $X$ and any $d$-dimensional subvariety $V\subset X$, we  set
\begin{equation}\label{eq:polynomialchargerestricted}
    Z_V(\cE)=  \left[\displaystyle\sum\nolimits_{j\geq 0} \rho_j c_1(L)^j\wedge U \wedge \ch(\cE)\right]^{(d,d)}\cup \cl(V)\in \C,
\end{equation}
where $\cl(V)\in H^{2n-2d}(X,\Q)$ is the cycle class of $V$.
\begin{definition}
    \label{def:Zpositivity}
    Let $\cE$ be a reflexive sheaf on $X$. Then $\cE$ is said to be \emph{$Z$-positive} if for any subvariety $V\subset X$ with $0<  \dim(V)<\dim(X)$, one has 
    \begin{equation}
    \label{eq:Z_inequalitysubvariety}
        \Im\left(\overline{Z_X(\cE)}Z_V(\cE)\right)>0.
    \end{equation}
    If in addition the inequality also holds for any $0$-dimensional cycle $V\subset X$, $\cE$ is said to be \emph{strongly $Z$-positive}.
\end{definition}
It is not difficult to see that $Z$-positivity appears as a special case of $P_\delta$-positivity.
This $Z$-positivity notion will be the one used in our applications of Theorem~\ref{theo:intro equiv} in Section~\ref{sec:toric applications}, and in some cases is a necessary and sufficient condition for the existence of subsolutions to the $Z$-critical equations (\cite{DMS,KScarpa}).

\subsubsection{A counter-example}
\label{sec:counter example P positive not P delta positive}
For general $\gamma_j$, the existence of a $P$-positive $P$-critical metric is not sufficient to deduce the $P_\delta$-positivity. We give a counter-example with the dHYM central charge outside of the supercritical regime in the end of this section.

Recall that a holomorphic Hermitian line bundle $(L,h)$ on $X$ satisfies the dHYM equation if and only if the eigenvalues $(\lambda_1,\ldots,\lambda_n)$ of its normalised curvature form with respect to the Kähler form $\omega$ verify
$$
\sum_{k = 1}^n \arccot(\lambda_k) = \textrm{constant}.
$$
In this case, we call the constant the \textit{phase} of $(L,h)$ \cite{CoShi}.

\begin{lemma}\label{LEM:Product of curves dHYM}
    Let $(X,\omega) = \prod_{k = 1}^n (C_k,\omega_k)$ be a product of curves and $\pi_k\colon X \rightarrow C_k$ the projections. If for all $k$, $L_k$ is a line bundle over $C_k$ and $L = \otimes_{k = 1}^n \pi_k^*L_k$ is a line bundle over $X$, then the HYM metric on $L$ is also dHYM with phase
    $$
    \theta = \sum_{k = 1}^n \arccot(\deg(L_k)).
    $$
\end{lemma}
\begin{proof}
Let for all $k$, $h_k$ be the HYM metric of $L_k$ over $C_k$. The associated curvature form verifies
$$
\cF(h_k) = \deg(L_k)\omega_k.
$$
Therefore, with $h = \otimes_{k = 1}^n \pi_k^*h_k$ on $L$, the curvature form satisfies
$$
\cF(h) = \sum_{k = 1}^n \pi_k^*\cF_k = \sum_{k = 1}^n \deg(L_k)\pi_k^*\omega_k.
$$
We immediately notice that the eigenvalues of $\cF(h)$ with respect to $\omega = \sum_{k = 1}^n \pi_k^*\omega_k$ are constant and equal $\deg(L_k)$. In particular, $h$ is both HYM and dHYM with phase
\begin{equation}
    \theta = \sum_{k = 1}^n \arccot(\deg(L_k)). \qedhere
\end{equation}
\end{proof}

The following is already well-known (see for example \cite[Sub-section 2.1]{CoShi}) but we give a proof for completeness.

\begin{lemma}\label{LEM:dHYM positive}
    All dHYM metrics are $P$-positive for the dHYM central charge.
\end{lemma}
\begin{proof}
Let $L \rightarrow (X,\omega)$ be a line bundle with a dHYM metric $h$. Set $(\lambda_1,\ldots,\lambda_n)$ the non-decreasing eigenvalues of the curvature form $\cF(h)$ of $L$ with respect to $\omega$. Let
$$
\theta = \sum_{k = 1}^n \arccot(\lambda_k)
$$
which is constant on $X$ since $h$ is a dHYM metric. Let us write at a given point $p \in X$,
$$
\omega = \sum_{k = 1}^n idz_k \wedge d\overline{z_k}, \qquad \cF(h) = \sum_{k = 1}^n \lambda_kidz_k \wedge d\overline{z_k}.
$$
When $1 \leq k \leq n$ is an integer, we denote by
$$
\widehat{idz_k \wedge d\overline{z_k}} = idz_1 \wedge d\overline{z_1} \wedge \cdots \wedge idz_{k - 1} \wedge d\overline{z_{k - 1}} \wedge idz_{k + 1} \wedge d\overline{z_{k + 1}} \wedge \cdots \wedge idz_n \wedge d\overline{z_n}.
$$
The formal derivative of the dHYM equation
$$
\cF \mapsto -\frac{1}{n!}\Im(e^{-i\theta}(\cF - i\omega)^n)
$$
with respect to $\cF$ is
$$
-\frac{1}{(n - 1)!}\Im(e^{-i\theta}(\cF - i\omega)^{n - 1}).
$$
At $\cF = \cF(h)$ the dHYM curvature form, it equals, at point $p$,
\begin{align*}
    &\ \ \ \ -\frac{1}{(n - 1)!}\Im(e^{-i\theta}(\cF(h) - i\omega)^{n - 1})\\
    & = -\frac{1}{(n - 1)!}\Im\left(e^{-i\theta}\left(\sum_{k = 1}^n (\lambda_k - i)idz_k \wedge d\overline{z_k}\right)^{n - 1}\right)\\
    & = -\sum_{k = 1}^n \Im\left(e^{-i\theta}\prod_{j = 1,j \neq k}^n \lambda_j - i\right)\widehat{idz_k \wedge d\overline{z_k}}\\
    & = -\sum_{k = 1}^n \Im\left(e^{-i\theta}\left(\prod_{j = 1,j \neq k}^n \sqrt{\lambda_j^2 + 1}\right)\exp\left(i\sum_{j = 1,j \neq k}^n \arccot(\lambda_j)\right)\right)\widehat{idz_k \wedge d\overline{z_k}}\\
    & = -\sum_{k = 1}^n \left(\prod_{j = 1,j \neq k}^n \sqrt{\lambda_j^2 + 1}\right)\sin\left(-\theta + \sum_{j = 1,j \neq k}^n \arccot(\lambda_j)\right)\widehat{idz_k \wedge d\overline{z_k}}\\
    & = \sum_{k = 1}^n \left(\prod_{j = 1,j \neq k}^n \sqrt{\lambda_j^2 + 1}\right)\sin(\arccot(\lambda_k))\widehat{idz_k \wedge d\overline{z_k}}\\
    & = \sum_{k = 1}^n \frac{1}{\sqrt{\lambda_k^2 + 1}}\left(\prod_{j = 1,j \neq k}^n \sqrt{\lambda_j^2 + 1}\right)\widehat{idz_k \wedge d\overline{z_k}}.
\end{align*}
We deduce that this $(n - 1,n - 1)$-form is positive and its eigenvalues with respect to
$$
\frac{\omega^{n - 1}}{(n - 1)!} = \sum_{k = 1}^n \widehat{idz_k \wedge d\overline{z_k}}
$$
are the $\frac{1}{\sqrt{\lambda_k^2 + 1}}\left(\prod_{j = 1,j \neq k}^n \sqrt{\lambda_j^2 + 1}\right) > 0$ for $1 \leq k \leq n$.
\end{proof}

Let $X = (\P^1)^3$ endowed with the sum $\omega$ of the Fubini--Study metrics on each $\P^1$ and $L = \cO_X(-1,0,1)$. By Lemma~\ref{LEM:Product of curves dHYM}, $L$ admits a dHYM metric $h$ with respect to $\omega$ with phase
$$
\theta = \arccot(-1) + \arccot(0) + \arccot(1) = \frac{3\pi}{2}.
$$
By Lemma~\ref{LEM:dHYM positive}, this metric is $P$-positive for the dHYM central charge.

However, we can compute that
$$
\int_{\P^1 \times \{0\} \times \{0\}} -\Im(e^{-i\theta}e^{c_1(L)}) = \int_{\P^1 \times \{0\} \times \{0\}} -c_1(L) = 1,
$$
$$
\int_{\{0\} \times \P^1 \times \{0\}} -\Im(e^{-i\theta}e^{c_1(L)}) = \int_{\{0\} \times \P^1 \times \{0\}} -c_1(L) = 0,
$$
$$
\int_{\{0\} \times \{0\} \times \P^1} -\Im(e^{-i\theta}e^{c_1(L)}) = \int_{\{0\} \times \{0\} \times \P^1} -c_1(L) = -1.
$$
We notice that even at constant dimension, the sign of the $P_V(L)$ vary despite $L$ admits a $P$-critical and $P$-positive metric.

\subsection{Toric case}
\label{sec:toric case}
Note that a toric variety satisfies the hypothesis of Theorem~\ref{theo:intro equiv}, and that Corollary~\ref{cor:intro} is then straightforward. While Theorem~\ref{theo:intro equiv} doesn't require the sheaf $\cE$ to be equivariant, equivariant sheaves over toric varieties enjoy a very simple description, that we recall here. We will then explain how this can be used to check $P_\delta$-positivity. We assume from now on and until the end of Section~\ref{sec:equivpos} that $X$ is a $n$-dimensional smooth projective toric variety. We will follow the notations from \cite{CLS}. In particular, $N$ stands for a rank $n$ lattice, $T=N\otimes_\Z \C^*$ is the associated torus, and $M=\Hom_\Z(N,\Z)$ is its dual lattice. Let $\Sigma$ be the smooth and complete fan of strongly convex rational polyhedral cones in $N_\R=N\otimes_\Z \R$ associated to $X$(see \cite[Chapter 3]{CLS}). The variety $X$ is then covered by the $T$-invariant affine toric varieties $U_\sigma=\Spec(\C[M\cap\sigma^\vee])$, for $\sigma\in \Sigma$.

\subsubsection{Toric sheaves}
\label{sec:toricsheaves}
A coherent sheaf $\cE$ on $X$ is said to be $T$-equivariant if there is an isomorphism
$$\varphi \colon \theta^*\cE \to \pi_2^*\cE$$
satisfying some cocycle condition, where $\theta \colon T \times X \to X$ and $\pi_2 \colon T\times X \to X$ stand for respectively the $T$-action and the projection on $X$ (see e.g.\ \cite{Perl}).
\begin{definition}
 \label{def:toricsheaves}
 A \emph{toric sheaf} on $X$ is a $T$-equivariant reflexive sheaf over $X$.
\end{definition}
From Klyachko's work \cite{Kl}, there is a fairly simple description of toric sheaves (see also \cite{Perl}). They are uniquely determined by \emph{families of filtrations}. If $\cE$ is a toric sheaf,  we denote $(E,E^\alpha(i))_{\alpha\in\Sigma(1),i\in\Z}$ the associated family of filtrations consisting of a finite dimensional complex vector space $E$ of dimension $\rank(\cE)$, and, for each ray $\alpha\in\Sigma(1)$, a bounded \emph{increasing} filtration $(E^\alpha(i))_{i\in\Z}$ of $E$.
The space $E$ is the fibre of $\cE$ over the identity element  $1\in T\subset X$, while the filtrations are obtained as follows. Being reflexive,  the  sections of $\cE$ extend over codimension $2$ subvarieties. If $\Sigma(j)$ is the set of $j$-dimensional cones in $\Sigma$, we set
$$
X_0=\bigcup_{\sigma\in\Sigma(0)\cup\Sigma(1)} U_\sigma.
$$ 
From \cite[Section 3.2]{CLS}, $X_0$ is the complement of $T$-orbits of co-dimension greater or equal to $2$, and thus $\cE=\iota_*(\cE_{\vert X_0})$, where $\iota \colon X_0 \to X$ is the inclusion. Equivariance implies that $\cE_{\vert X_0}$ is entirely characterised by the sections $\Gamma(U_\sigma,\cE)$, for $\sigma\in\Sigma(0)\cup\Sigma(1)$. If $\sigma=\lbrace 0 \rbrace$, $U_{\lbrace 0\rbrace}=T$, and
$$
\Gamma(U_{\lbrace 0\rbrace},\cE)=E\otimes_\C \C[M].
$$
 Then, if $\alpha\in \Sigma$ is a ray (i.e.\ a one-dimensional cone), $\Gamma(U_{\alpha}, \cE)$ is graded by $$M/( M\cap \alpha^\perp)\simeq\Z.$$
 As $T$ is a dense open subset of $U_\alpha$, the restriction map $\Gamma(U_\alpha,\cE)\to\Gamma(U_{\lbrace 0\rbrace},\cE)$ is injective and induces a $\Z$-filtration 
$$
  \ldots \subset E^\alpha(i-1)\subset E^\alpha(i) \subset \ldots \subset E
$$
such that one has
 \begin{equation*}
  \label{eq:sheaf from family of filtrations first}
  \Gamma(U_{\alpha}, \cE)=\bigoplus_{m\in M} E^\alpha(\langle m,u_\alpha\rangle)\otimes \chi^m,
 \end{equation*}
 where we denote by $u_\alpha$ the primitive generator of $\alpha$ and $\langle \cdot,\cdot\rangle$ the duality pairing. In the other direction, one recovers $\cE$ by setting for each $\sigma\in\Sigma$: 
\begin{equation}
  \label{eq:sheaf from family of filtrations}
  \Gamma(U_{\sigma}, \cE)\coloneqq \bigoplus_{m\in M} \bigcap_{\alpha\in\sigma(1)} E^\alpha(\langle m,u_\alpha\rangle)\otimes \chi^m.
 \end{equation}

\begin{example}
\label{ex:tangent}
 The $T$-action $\theta \colon T\times X \to X$ induces naturally an equivariant structure on the sheaf of K\"ahler differentials $\Omega_X^1$ given by the composition: 
$$
\theta^*\Om^1_X \xrightarrow{d\theta} \Om^1_{T\times X}\xrightarrow{\simeq} \pi_1^*\Om^1_{T}\oplus\pi_2^*\Om^1_X \xrightarrow{\mathrm{pr}_2} \pi_2^*\Om^1_X
$$
where $\pi_i$ (resp. $\mathrm{pr}_2$) is the projection on the $i$-th factor of $T\times X$ (resp. on the second factor of $\pi_1^*\Om^1_{T}\oplus\pi_2^*\Om^1_X$). By duality, $\cT_X$ is a toric sheaf. The family of filtrations $(E^\alpha(\bullet))_{\alpha\in\Sigma(1)}$ for $\cT_X$ is given by \cite[Example 2.3(5) on page 350]{Kl}:
$$
E^\alpha(i)=\left\{ \begin{array}{ccc}
                  \lbrace 0 \rbrace & \textrm{ if } & i\leq -2 \\
                  \mathrm{Span}( u_\alpha) & \textrm{ if } & i=-1 \\
                  N\otimes_\Z \C & \textrm{ if } & i\geq 0.
                 \end{array} 
\right.
$$
\end{example}

\begin{example}
\label{ex:line bundles}
 By the orbit-cone correspondence \cite[Chapter 3]{CLS}, $T$-invariant  reduced and irreducible divisors on $X$ are in bijection with rays of $\Sigma$. Hence, any $T$-invariant Cartier divisor $D\subset X$ can be written
 $$
 D=\sum_{\alpha\in\Sigma(1)} a_\alpha D_\alpha
 $$
 where $D_\alpha\subset X$ is the orbit closure associated to $\alpha\in\Sigma(1)$. For example, the canonical divisor reads (c.f.\ \cite[Chapter 8]{CLS})
 $$
 K_X=-\sum_{\alpha\in\Sigma(1)} D_\alpha.
 $$
The family of filtrations associated to $\cO(D)$ for $D=\sum_{\alpha\in\Sigma(1)} a_\alpha D_\alpha$ is given by $E=\C$ and for each $\alpha\in\Sigma(1)$, and each $i\in \Z$,
  \begin{equation*}
E^\alpha(i)=
\left\{ 
\begin{array}{ccc}
0 & \mathrm{ if } & i < -a_\alpha\\
\C & \mathrm{ if } & i \geq -a_\alpha.
\end{array}
\right.
\end{equation*}
\end{example}
When testing $P$-stability, one considers saturated reflexive subsheaves of a given one. In the toric case, those are described in \cite[Lemma 2.15]{NaTip} as follows: 
\begin{lemma}
\label{lem:sat subsheaves equiv case}
If  $(E,E^\rho(\bullet))_{\rho\in\Sigma(1)}$ stands for the family of filtrations of a toric sheaf $\cE$, saturated equivariant reflexive subsheaves of $\cE$ are associated to families of filtrations of the form $(F, F\cap E^\rho(i))_{\rho\in\Sigma(1),i\in\Z}$ for vector subspaces $F\subsetneq E$.
\end{lemma}

\subsubsection{\texorpdfstring{$P$}{P}-slopes}
\label{sec:slopes}
From \cite{Kl} and \cite[Proposition 3.16]{Koo}, if $X$ is smooth, the Chern character of any toric sheaf admits a simple combinatorial description in terms of the family of filtrations. The Chern character of the tangent bundle can be recovered as follows. 
\begin{example}[Chern character of the tangent bundle]
\label{ex:chern-of-tangent}
The tangent bundle is given by the following
exact sequence
\begin{equation}\label{eq:euler-sequence}
    0 \longrightarrow \Hom(\Pic(X), \Z) \otimes \cO_X \longrightarrow
    \bigoplus_{\alpha \in \Sigma(1)} \cO_X(D_\alpha) \longrightarrow \cT_X \longrightarrow 0.
\end{equation}
By \cite[Proposition 13.1.2]{CLS}, we have
$$
c(\cT_X) = \prod_{\alpha \in \Sigma(1)}(1 + D_\alpha) = \sum_{\sigma \in \Sigma} V(\sigma).
$$
Using~\eqref{eq:euler-sequence}, for any $k \in \{1, \ldots, n \}$,
$$
\ch_k(\cT_X) = \dfrac{1}{k!} \sum_{\alpha \in \Sigma(1)} D_{\alpha}^k.
$$
\end{example}

Recall that the cycle class map provides an isomorphism between the rational Chow ring and the rational cohomology ring of $X$ (\cite[Theorem 12.5.3]{CLS}), hence we will often omit the cycle class map $\mathrm{cl}$ in our notations, when dealing with intersection theory on toric varieties. Let $P\in H^\bullet(X,\R)[x]$ and $\delta\in\lbrace -1,+1 \rbrace^n$ as before. 
From  \cite[Lemma 12.5.1]{CLS}, we may assume that 
$$
 P(x)=\displaystyle\sum_{\sigma\in\Sigma} v_\sigma\,V(\sigma)\wedge\frac{x^{\codim{(\sigma)}}}{\codim{(\sigma)}!}\in H^{\bullet}(X,\RR)[x]
$$
for some real numbers $(v_\sigma)_{\sigma\in\Sigma}$.  We deduce that for any toric sheaf $\cE$ on $X$, and for any invariant cycle $V\subset X$, the number $P_V(\cE)$ can be computed from the intersection theory of torus invariant cycles, the latter being described, for example, in \cite[Section 12.5]{CLS}. By Corollary~\ref{cor:intro}, $P_\delta$-positivity can be verified effectively on toric varieties. We illustrate this in Section~\ref{sec:toric applications}.

\subsection{Equivariant \texorpdfstring{$P$}{P}-stability}
\label{sec:Zstability}
In this Section, we consider an \emph{equivariant} version of $P$-stability. Fix a polynomial $P\in H^\bullet(X,\R)[x]$ as before.
\begin{definition}\label{def:equivariant}
    Let $X$ be a projective $T$-variety, for $T$ a torus.  A $T$-equivariant reflexive sheaf $\cE$ on $X$ is  \emph{equivariantly $P$-stable (resp. $P$-semistable)} if 
    \begin{equation}\label{eq:ineqalitysheaves_equiv}
        P_X(\cF)<P_X(\cE)\  (\mathrm{resp.} \leq P_X(\cE)).
    \end{equation}
    for any $T$-equivariant saturated subsheaf $\cF\subset \cE$ with $0<\rank(\cF)<\rank(\cE)$.
\end{definition}
We can also consider its asymptotic version: if $P_r$ is an asymptotic polynomial equation as in Section~\ref{sec:asymptotic_regimes}, we can consider the associated family $\R\to H^\bullet(X,\R)[x]$ of polynomials. We will then say that $\cE$ is asymptotically (equivariantly) $P_r$-stable (resp. semistable) if  for any (equivariant) saturated strict subsheaf $\cF\subset \cE$, 
$$P_{r,X}(\cF)<P_{r,X}(\cE)$$
(resp. $P_{r,X}(\cF)\leq P_{r,X}(\cE)$) for $r$ large enough.

On $T$-varieties, slope and Gieseker stabilities are equivalent to there equivariant versions (c.f.\ \cite{Koo}). This was successfully used on toric varieties in mathematical physics where slope stable sheaves are ubiquitous in string theories \cite{KnuSha98}, or in mathematics as a toy model to test new constructions of slope stable sheaves (e.g.\ \cite{NaTip,ClaTip,ClaNapTip,Tip}) and to obtain classification results for slope stable tangent sheaves (e.g.\ \cite{DDK,HNS,Nap}).

Let's recall briefly the ingredients in Kool's proof that equivariant Gieseker stability matches with Gieseker stability for torus equivariant reflexive sheaves. In the semi-stable case, from uniqueness of the Harder--Narasimhan filtration (refer to \cite[Chapter 1]{HuLe} for definitions), one obtains that this filtration is by equivariant reflexive sheaves. Hence, equivariant semi-stability implies semi-stability. In the stable case, the Jordan--H\"older filtration may not be unique, but the reflexive hull of its graded object is. By the semi-stable case, if $\cE$ is equivariantly stable, it is semi-stable, and the graded object of its Jordan--H\"older filtration is equivariant by uniqueness. Then, one has to show that each of its stable components is equivariant. This is achieved in \cite[Proposition 3.19]{Koo}, using general facts about representation theory of algebraic tori and $T$-varieties, and also the fact that there are no morphism between two non-isomorphic Gieseker stable sheaves of the same slope.

The main issue to extend Kool's result to $P$-stability is the non-existence in general of Harder--Narasimhan and Jordan--H\"older filtrations. This is mainly due to the fact that considering the saturation of a subsheaf doesn't increase the $P$-slope. Nevertheless, we hope that a combination of $P$-stability and a stronger version of $P$-positivity, called $P$-positivity with respect to quotients in \cite{KScarpa}, could be equivalent to the combination of their equivariant versions.

In the meantime, a class of asymptotic polynomial equations called \emph{adapted to torsion-free sheaves} was identified in \cite{Delloque_adapted}:
\begin{definition}
    \label{def:adapted}
    An asymptotic polynomial equation $P_r$ is adapted to torsion-free sheaves if for any torsion-free sheaf $\cE$ and any coherent sub-sheaf $\cF\subset \cE$ such that $\cE/\cF$ is torsion, one has 
    $$
P_{r,X}(\cF)< P_{r,X}(\cE)    
    $$
    for $r$ large enough.
\end{definition}
This class of adapted polynomials generalizes Gieseker stability and the stability conditions  considered in \cite{Toma_semistab_ampleclasses}, as discussed in Section~\ref{sec:asymptotic_regimes} (c.f.\ \cite[Example 2.4]{Delloque_adapted}).

In asymptotic regimes, for asymptotic polynomial equations adapted to torsion-free sheaves, using the results from \cite{Delloque_adapted}, we obtain Theorem~\ref{theo:intro equiv Zstab asymptotic}. Indeed, the proof is the same as the one given by Kool in \cite[Proposition 3.19]{Koo}, once we have the see-saw property of $P$-slopes (that follows from additivity of the Chern character), and the existence and uniqueness of the aforementioned filtrations \cite[Theorems 2.9 and 2.10 ]{Delloque_adapted}. 

Using Theorem~\ref{theo:intro equiv Zstab asymptotic}, together with the description of toric sheaves and their saturated subsheaves in Section~\ref{sec:toricsheaves} and Lemma~\ref{lem:sat subsheaves equiv case}, we may test efficiently $P$-stability for toric sheaves for polynomials adapted to torsion-free sheaves. We give an explicit example of a $P$-stable tangent bundle
for a $3$-dimensional toric manifold in Section~\ref{sec:application asymptotic case}.

\section{P-positivity and blow-ups}
\label{sec:blowups}

Let $\cE \rightarrow X$ be a vector bundle on a compact manifold of dimension $n \geq 2$. Consider the map
$$
V \mapsto P_V(\cE) = \left(\sum_{j = 0}^n [\gamma_j] \smile \Chern(\cE)\right) \frown [V] = \left(\sum_{j = 0}^n [\gamma_j] \smile \Chern(\cE) \smile \cl(V)\right) \frown [X].
$$
Here, each $\gamma_j$ is a closed real $(n - j,n - j)$ form on $X$. Notice that this map naturally extends to a morphism of groups $H^{*,*}(X,\RR) \longrightarrow \RR$ as,
$$
c \mapsto P_c(\cE) = \left(\sum_{j = 0}^n [\gamma_j] \smile \Chern(\cE) \smile c\right) \frown [X].
$$
Assume moreover that,
$$
P_X(\cE) = \left(\sum_{j = 0}^n [\gamma_j] \smile \Chern(\cE)\right) \frown [X] = 0.
$$

\subsection{Cohomology and uniform $P$-positivity}

Let us start with a simple characterisation of strong $P_\delta$-positivity.

\begin{lemma}[{\cite[Remark 2.5]{KScarpa}}]\label{LEM:P positivité forte}
    $\cE$ is strongly $P_\delta$-positive if and only if $\cE$ is $P_\delta$-positive and $\delta_0\gamma_n > 0$.
\end{lemma}
\begin{proof}
If $V \subset X$ has dimension $0$, it is a non-empty finite collection of points. We have,
$$
P_V(\cE) = \left(\sum_{j = 0}^n [\gamma_j] \smile \Chern(\cE)\right) \frown [V] = [\gamma_n]\rk(\cE)\#(V).
$$
We immediately deduce the result.
\end{proof}

Then, we blow-up a point of $X$ and we recall some cohomological facts on the blow-up.

Let $W \subset X$ be a singleton in $X$, seen as an irreducible sub-variety of dimension $0$. Let $\pi \colon \tilde{X} \rightarrow X$ be the blow-up of $W$. Let $D$ be the exceptional divisor so we have the following commutative diagram.
$$
\begin{tikzcd}
    D \ar["\iota_D"']{d}\ar["\pi"]{r} & W \ar["\iota_W"]{d}\\
    \tilde{X} \ar["\pi"']{r} & X
\end{tikzcd}
$$
Moreover, as $W$ is a point, $D$ is isomorphic to $\PP^{n - 1}$. Let $(\varepsilon_0,\ldots,\varepsilon_n)$ be real numbers. When $c \in H^{*,*}(\tilde{X},\RR)$, we consider the \emph{pulled-back polynomial}: 
$$
\tilde{P}_{\varepsilon,c}(\pi^*\cE) \coloneqq \left(\sum_{j = 0}^n ([\pi^*\gamma_j] + \varepsilon_j\cl(D)^{n - j}) \smile \Chern(\cE) \smile c\right) \frown [X].
$$
We want to find conditions on $\cE$ and the $\varepsilon_k$ for $\pi^*\cE$ to be $\tilde{P}_{\varepsilon,\delta}$-positive.

First of all, let us recall the cohomological structure of $\tilde{X}$ in function of the one of $X$. For all integer $1 \leq k \leq n - 1$, there is an isomorphism \cite[Theorem 7.31]{Voisin},
$$
H^{n - k,n - k}(X,\RR) \oplus \RR \cong H^{n - k,n - k}(\tilde{X},\RR).
$$
The injection
$$
H^{n - k,n - k}(X,\RR) \longrightarrow H^{n - k,n - k}(\tilde{X},\RR)
$$
is $\pi^*$. The projection
$$
H^{n - k,n - k}(\tilde{X},\RR) \longrightarrow H^{n - k,n - k}(X,\RR)
$$
is $\pi_*$. The injection
$$
\RR \longrightarrow H^{n - k,n - k}(\tilde{X},\R)
$$
sends $a$ on $(-1)^{n - k}a\,\cl(D)^{n - k}$. The reason why we normalise by $(-1)^{n - k}$ shall appear later. Let us introduce the last projection
$$
l \colon H^{n - k,n - k}(\tilde{X},\RR) \longrightarrow \RR.
$$
If $c \in H^{n - k,n - k}(\tilde{X},\R)$ with $1 \leq k \leq n - 1$, we can write
\begin{equation}
    c = \pi^*\pi_*c + (-1)^{n - k}l(c)\cl(D)^{n - k}.
\end{equation}
When we compute the cap product with $\cl(D)$, we obtain,
\begin{equation}\label{EQ:Expression l(c)}
    c \smile \cl(D) = (-1)^{n - k}l(c)\cl(D)^{n - k + 1}.
\end{equation}
It entirely characterises $l$. In particular, by taking again the cap with $\cl(D)^{k - 1}$ then integrating over $X$ and using the equality $\cl(D)^n \frown [\tilde{X}] = (-1)^{n - 1}$, we obtain,
$$
l(c) = (-1)^{k - 1}(c \smile \cl(D)^k) \frown [\tilde{X}].
$$

In order to study the $\tilde{P}_{\varepsilon,\delta}$-positivity of $\pi^*\cE$ in function of the $P_\delta$-positivity of $\cE$, we need to introduce a slightly stronger condition. For all $0 \leq k \leq n$, let the \textit{effective cone} of dimension $k$,
$$
\Eff_k(X) = \set*{\textstyle\sum\nolimits_{i \in I} a_i\cl(V_i) \suchthat I \textrm{ finite}, V_i \subset X \textrm{ of dimension } k, a_i \geq 0 }.
$$
We define the \textit{pseudo-effective cone} of dimension $k$ as the closure $\overline{\Eff_k(X)}$ of $\Eff_k(X)$ in $H^{n - k,n - k}(X,\RR)$. By linearity, we easily notice that if $\cE$ is $P_\delta$-positive, then, for all $1 \leq k \leq n - 1$ and all $c \in \Eff_k(X)\backslash\{0\}$, $\delta_kP_c(\cE) > 0$. Let $\|\cdot\|$ be any norm on the finite dimensional space $H^{n - k,n - k}(X,\RR)$.

\begin{proposition}\label{PRO:P-positivité uniforme}
    Let $1 \leq k \leq n - 1$. We have equivalence between,
    \begin{enumerate}
        \item For all $c \in \overline{\Eff_k(X)}\backslash\{0\}$, $\delta_kP_c(\cE) > 0$.
        \item There is a positive constant $\epsilon$ such that for all $V \subset X$ of dimension $k$, $\delta_kP_V(\cE) \geq \epsilon\|\cl(V)\|$.
    \end{enumerate}
\end{proposition}
\begin{proof}
Assume first that the first point is verified. Let,
$$
f \colon \left\{
\begin{array}{rcl}
    \overline{\Eff_k(X)}\backslash\{0\}/\RR_+^* & \longrightarrow & \R\\
    \R_+^*c & \longmapsto & \displaystyle \frac{\delta_kP_c(\cE)}{\|c\|}
\end{array}
\right.
$$
$f$ is well-defined and continuous. Moreover, by hypothesis, it takes values in $\R_+^*$. Since $H^{n - k,n - k}(X,\RR)$ is finite dimensional, $\overline{\Eff_k(X)}\backslash\{0\}/\RR_+^*$ is compact. Let $\epsilon > 0$ be the minimum of $f$. For all $V \subset H^{n - k,n - k}(X,\RR)$, $\cl(V) \in \overline{\Eff_k(X)}\backslash\{0\}$ so,
$$
\delta_kP_V(\cE) = f(\R_+^*\cl(V))\|\cl(V)\| \geq \epsilon\|\cl(V)\|.
$$

Reciprocally, assume that the second point is verified. Then, for all $c \in \Eff_k(X)$, we can write
$$
c = \sum_{i \in I} a_i\cl(V_i),
$$
with $I$ finite and non-empty and each $a_i$ is positive. We have
$$
\delta_kP_c(\cE) = \sum_{i \in I} a_i\delta_kP_{V_i}(\cE) \geq \epsilon\sum_{i \in I} a_i\|\cl(V_i)\| \geq \epsilon\|c\|.
$$
This inequality remains valid for $c \in \overline{\Eff_k(X)}$. In particular, if $c \in \overline{\Eff_k(X)}\backslash\{0\}$, $\delta_kP_c(\cE) > 0$.
\end{proof}

\begin{definition}
    If the equivalent conditions of Proposition~\ref{PRO:P-positivité uniforme} are satisfied for all $1 \leq k \leq n - 1$ (resp. $0 \leq k \leq n - 1$), we say that $\cE$ is uniformly $P_\delta$-positive (resp. strongly uniformly $P_\delta$-positive).
\end{definition}
\begin{remark}
Since $\Eff_0(X) = \overline{\Eff_0(X)}$ is closed (it is a half line in the real line $H^{n,n}(X,\RR)$), strong uniform $P_\delta$-positivity is equivalent to uniform $P_\delta$-positivity and strong $P_\delta$-positivity.
\end{remark}

Of course, uniform $P_\delta$-positivity implies $P_\delta$-positivity. In the case of the $J$-equation for example, Chen showed that a line bundle is uniformly $P$-positive if and only if it admits a solution to the $J$-equation \cite[Theorem 1.1]{Che21}. Song later improved this correspondence by removing the hypothesis of uniformity \cite[Corollary 1.2]{Song20}. In particular, it proves that these two notions are equivalent for the central charge of the $J$-equation when all $\delta_k$ equal $1$.

We expect the uniformity assumption to be superfluous in a lot of cases but not all, since the effective cones may not be closed. However, in some cases like Fano surfaces and projective spaces, the effective cones are closed \cite[Theorem 1.5.33]{Lazarsfeld} hence $P_\delta$-stability is equivalent to uniform $P_\delta$-stability for any $P$ and $\delta$.

We end this sub-section by describing the pseudo-effective cones of $\tilde{X}$ in function of the ones of $X$.

\begin{lemma}\label{LEM:Cones effectifs pull-back}
    Let $1 \leq k \leq n - 1$ and $c \in H^{n - k,n - k}(\tilde{X},\RR)$. Recall that $c = \pi^*\pi_*c + (-1)^{n - k}l(c)\cl(D)^{n - k}$. There is a constant $C$ independent from $c$ such that,
    \begin{enumerate}
        \item If $\pi_*c \in \overline{\Eff_k(X)}$ and $l(c) = 0$, then $c \in \overline{\Eff_k(\tilde{X})}$.
        \item If $\pi_*c = 0$ and $l(c) < 0$, then $c \in \Eff_k(\tilde{X})$.
        \item If $c \in \overline{\Eff_k(\tilde{X})}$, then $\pi_*c \in \overline{\Eff_k(X)}$ and $l(c) \leq C\|\pi_*c\|$.
    \end{enumerate}
\end{lemma}
\begin{proof}
For the first point, we may assume without loss of generality that $\pi_*c = \cl(V)$ with $V \subset X$ irreducible of dimension $k$. Then, $c = \pi^*\cl(V)$. Let
$$
\tilde{V}_1 = \overline{\pi^{-1}(V)\backslash D}.
$$
It has dimension $k$ and we have $\pi_*\cl(\tilde{V}_1) = \cl(\pi(\tilde{V}_1)) = \cl(V)$ so
\begin{equation}\label{EQ:Expression c}
    \cl(\tilde{V}_1) = \pi^*\cl(V) + (-1)^{n - k}l(\cl(\tilde{V}_1))\cl(D)^{n - k} = c + (-1)^{n - k}l(\cl(\tilde{V}_1))\cl(D)^{n - k}.
\end{equation}
$\tilde{V}_1$ intersects properly $D$. We deduce that their intersection $\tilde{V}_1 \cap D$ is a sub-variety of $D$ of dimension $k - 1$. Recall that $D$ is isomorphic to $\P^{n - 1}$. Let $H_D$ be a hyperplane class. We have $\tilde{V}_1 \cap D \in \iota_{D*}(dH_D^{n - k})$ where $d \geq 0$ is the degree of this variety and $\iota_D \colon D \rightarrow X$ is the inclusion. We deduce that,
\begin{equation}\label{EQ:Intersection 1}
    \cl(\tilde{V}_1) \smile \cl(D) = \cl(\tilde{V}_1 \cap D) = \iota_{D*}(dH_D^{n - k}) = (-1)^{n - k}d\,\cl(D)^{n - k + 1}.
\end{equation}
Now, by~\eqref{EQ:Expression l(c)}, we have,
\begin{equation}\label{EQ:Intersection 2}
    \cl(\tilde{V}_1) \smile \cl(D) = (-1)^{n - k}l(\cl(\tilde{V}_1))\cl(D)^{n - k + 1}.
\end{equation}
From~\eqref{EQ:Intersection 1} and~\eqref{EQ:Intersection 2}, we deduce that $l(\cl(\tilde{V}_1)) = d \geq 0$. Let $\tilde{V}_2 \subset D$ be a sub-variety of dimension $k$ and degree $d$. We have,
\begin{equation}\label{EQ:Expression cl(V2)}
    \cl(\tilde{V}_2) = \iota_{D*}(dH_D^{n - k - 1}) = (-1)^{n - k - 1}d\,\cl(D)^{n - k}.
\end{equation}
We deduce from~\eqref{EQ:Expression c} and~\eqref{EQ:Expression cl(V2)} that
\begin{equation}
\begin{split}
    c =& \cl(\tilde{V}_1) - (-1)^{n - k}l(\cl(\tilde{V}_1))\cl(D)^{n - k} = \cl(\tilde{V}_1) - (-1)^{n - k}d\,\cl(D)^{n - k} =\\
    =& \cl(\tilde{V}_1) + \cl(\tilde{V}_2).
\end{split}
\end{equation}
Therefore, $c \in \Eff_k(\tilde{X})$.

For the second point, we deduce from the computation of~\eqref{EQ:Expression cl(V2)} that if $\tilde{V} \subset D$ has dimension $k$ and degree $1$, $\cl(\tilde{V}) = -(-1)^{n - k}\cl(D)^{n - k}$ so,
$$
\pi_*\cl(\tilde{V}) = 0, \qquad l(\cl(\tilde{V})) = -1 < 0.
$$
From this, we easily deduce the second point.

For the third point, if $c = \cl(\tilde{V})$ for some $\tilde{V} \subset \tilde{X}$ irreducible of dimension $k$, $\pi_*c = 0$ if $\tilde{V} \subset D$, $\pi_*c = \cl(\pi(\tilde{V}))$ else. We deduce easily that for all $c \in \overline{\Eff_k(\tilde{X})}$, $\pi_*c \in \overline{\Eff_k(X)}$. Finally, recall that the pseudo-effective cones don't contain any line. Therefore, for all $c \in \overline{\Eff_k(\tilde{X})}$, if $\pi_*c = 0$, then $l(c) \leq 0$. We deduce that the function
$$
\begin{array}{rcl}
    \overline{\Eff_k(\tilde{X})}\backslash\{0\}/\R_+^* & \longrightarrow & \R \cup \{-\infty\}\\
    c & \longmapsto & \displaystyle \frac{l(c)}{\|\pi_*c\|}
\end{array}
$$
is well-defined and continuous. By an argument of compactness, this function reaches a maximum $C$ and for all $c \in \overline{\Eff_k(\tilde{X})}$, $l(c) \leq C\|\pi_*c\|$.
\end{proof}

\subsection{Pull-back of a bundle}

We give here a necessary and sufficient condition for $\pi^*\cE$ to be uniformly $\tilde{P}_{\varepsilon,\delta}$-positive when $(\varepsilon_0,\ldots,\varepsilon_n)$ are small enough. Let us start by computing $\tilde{P}_{\varepsilon,c}(\pi^*\cE)$ in function of $\varepsilon$.

\begin{lemma}\label{LEM:Calcul P tilde}
    We have,
    $$
    \tilde{P}_{\varepsilon,\tilde{X}}(\pi^*\cE) = (-1)^{n - 1}\rk(\cE)\varepsilon_0 + \varepsilon_n\Chern_n(\cE) \frown [X]
    $$
    and for all $1 \leq k \leq n - 1$ and all class $c \in H^{n - k,n - k}(X,\RR)$,
    $$
    \tilde{P}_{\varepsilon,c}(\pi^*\cE) = P_{\pi_*c}(\cE) + \varepsilon_n(\Chern_k(\cE) \smile \pi_*c) \frown [X] + (-1)^{k - 1}\rk(\cE)l(c)\varepsilon_{n - k}.
    $$
\end{lemma}
\begin{proof}
Let us start with $\tilde{X}$. Recall that for all classes $c \in H^{i,i}(X,\RR)$ with $i > 0$ and all $j > 0$, $\pi^*c \smile \cl(D)^j = 0$.
\begin{equation}
    \begin{split}
        \tilde{P}_{\varepsilon,\tilde{X}}&(\pi^*\cE)  = \left(\sum_{j = 0}^n (\pi^*[\gamma_j] + \varepsilon_j\cl(D)^{n - j}) \smile \Chern_j(\pi^*\cE)\right) \frown [\tilde{X}]\\
    & = \pi^*\left(\sum_{j = 0}^n [\gamma_j] \smile \Chern_j(\cE)\right) \frown [\tilde{X}] + \left(\sum_{j = 0}^n \varepsilon_j\cl(D)^{n - j} \smile \pi^*\Chern_j(\cE)\right) \frown [\tilde{X}]\\
    & = \left(\sum_{j = 0}^n [\gamma_j] \smile \Chern_j(\cE)\right) \frown [X] + \rk(\cE)\varepsilon_0\cl(D)^n \frown [\tilde{X}] + \varepsilon_n\pi^*\Chern_n(\cE) \frown [\tilde{X}]\\
    & = (-1)^{n - 1}\rk(\cE)\varepsilon_0 + \varepsilon_n\Chern_n(\cE) \frown [X].
    \end{split}
\end{equation}
The last equality comes from the assumption $P_X(\cE) = 0$ and the equality $\cl(D)^n \frown [\tilde{X}] = (-1)^{n - 1}$.

Then, let $1 \leq k \leq n - 1$ be an integer and $c \in H^{n - k,n - k}(C,\RR)$.
\begin{equation}
\begin{split}
    \tilde{P}_{\varepsilon,c}(\pi^*\cE)  =  \sum_{j = n - k}^n \Big( 
    (\pi^*[\gamma_j] + \varepsilon_j\cl(D)^{n - j}) \smile \Chern_{j - n + k}(\pi^*\cE)& \\
    \smile(\pi^*\pi_*c + (-1)^{n - k}l(c)\cl(D)^{n - k}) 
    &\Big) \frown [\tilde{X}]
\end{split}
\end{equation}
and splitting the sum we obtain
\begin{equation}
\begin{split}
    \tilde{P}_{\varepsilon,c}(\pi^*\cE)  = &  \pi^*\left(\sum_{j = n - k}^n [\gamma_j] \smile \Chern_{j - n + k}(\cE) \smile \pi_*c\right) \frown [\tilde{X}] \\
    &  + \left(\sum_{j = n - k}^n \varepsilon_j\cl(D)^{n - j} \smile \pi^*\Chern_{j - n + k}(\cE) \smile \pi^*\pi_*c\right) \frown [\tilde{X}]\\
    + (-1)^{n - k} & l(c)\left(\sum_{j = n - k}^n \pi^*[\gamma_j] \smile \pi^*\Chern_{j - n + k}(\cE) \smile \cl(D)^{n - k}\right) \frown [\tilde{X}]\\
      + (-1)^{n - k} & l(c)\left(\sum_{j = n - k}^n \varepsilon_j\cl(D)^{n - j} \smile \pi^*\Chern_{j - n + k}(\cE) \smile \cl(D)^{n - k}\right) \frown [\tilde{X}] \\
     = &  \left(\sum_{j = n - k}^n [\gamma_j] \smile \Chern_{j - n + k}(\cE) \smile \pi_*c\right) \frown [X] \\
    &  + \varepsilon_n\pi^*(\Chern_k(\cE) \smile \pi_*c) \frown [\tilde{X}] + (-1)^{n - k}l(c)\varepsilon_{n - k}(-1)^{n - 1}\rk(\cE)  \\
     = &  P_{\pi_*c}(\cE) + \varepsilon_n(\Chern_k(\cE) \smile \pi_*c) \frown [X] + (-1)^{k - 1}\rk(\cE)l(c)\varepsilon_{n - k}.
\end{split}
\end{equation}
\end{proof}

For the rest of this section, we study how (strong) uniform $P_\delta$-positivity is preserved by pull-backs. For all $1 \leq k \leq n$, we endow the spaces $H^{n - k,n - k}(\tilde{X},\RR)$ with the following norms,
$$
\|c\| = \|\pi_*c\| + |l(c)|.
$$
The choice of the norms doesn't change anything since these spaces have finite dimensions. We can now conclude the proof of Theorem~\ref{THE:P positivité pull-back}.

\begin{proof}[Proof of Theorem~\ref{THE:P positivité pull-back}]\ 

\noindent\framebox{1 $\Rightarrow$ 3} By Lemma~\ref{LEM:Calcul P tilde}, for any $\varepsilon$ we have $\tilde{P}_{\varepsilon,\tilde{X}}(\pi^*\cE) = 0$ if and only if
\begin{equation}\label{eq:eps0_normalisation}
    \varepsilon_0 = (-1)^n\frac{\Chern_n(\cE) \frown [X]}{\rk(\cE)}\varepsilon_n.
\end{equation}
As $\pi^*\cE$ is (strongly) $\tilde{P}_{\varepsilon,\delta}$-positive,~\eqref{eq:eps0_normalisation} holds by definition.

Let $1 \leq k \leq n - 1$. By Lemma~\ref{LEM:Cones effectifs pull-back}, the class $c \in H^{n - k,n - k}(\tilde{X},\RR)$ such that $\pi_*c = 0$ and $l(c) = -1$ lies in $\Eff_k(\tilde{X})$ and by Lemma~\ref{LEM:Calcul P tilde},
\begin{equation}
    \begin{split}
        \tilde{P}_{\varepsilon,c}(\pi^*\cE) = & P_{\pi_*c}(\cE) + \varepsilon_n(\Chern_k(\cE) \smile \pi_*c) \frown [X] + (-1)^{k - 1}\rk(\cE)l(c)\varepsilon_{n - k}\\
        = & (-1)^k\rk(\cE)\varepsilon_{n - k}.
    \end{split}
\end{equation}
By assumption, this quantity has the sign of $\delta_k$ is positive hence $(-1)^k\delta_k\varepsilon_{n - k} > 0$. Notice that this holds for any $\varepsilon$ and even without any $P_\delta$-positivity assumption on $\cE$.

\smallskip

\noindent\framebox{2 $\Rightarrow$ 1} Trivial.

\smallskip

\noindent\framebox{3 $\Rightarrow$ 2} Still by Lemma~\ref{LEM:Calcul P tilde}, we have $\tilde{P}_{\varepsilon,\tilde{X}}(\pi^*\cE) = 0$. Let $c \in \overline{\Eff_k(\tilde{X})}$. By Lemma~\ref{LEM:Cones effectifs pull-back}, $\pi_*c \in \overline{\Eff_k(X)}$ and there is a positive constant $C$ such that
\begin{equation}\label{EQ:Borne l(c)}
    l(c) \leq C\|\pi_*c\|.
\end{equation}
Moreover, by Proposition~\ref{PRO:P-positivité uniforme}, there is a positive constant $\epsilon$ such that
\begin{equation}\label{EQ:Borne P_c(E)}
    \delta_kP_{\pi_*c}(\cE) \geq \epsilon\|\pi_*c\|,
\end{equation}
since $\pi_*c \in \overline{\Eff_k(X)}$. Finally, by a continuity argument, there is a constant $C'$ only depending on the Chern character of $\cE$ such that
\begin{equation}\label{EQ:Borne Ch smile c}
    \left\lvert(\Chern_k(\cE) \smile \pi_*c) \frown [X]\right\rvert \leq C'\|\pi_*c\|.
\end{equation}
From Lemma~\ref{LEM:Calcul P tilde} we deduce that
\begin{equation}
    \delta_k\tilde{P}_{\varepsilon,c}(\pi^*\cE) = \delta_kP_{\pi_*c}(\cE) + \delta_k\varepsilon_n(\Chern_k(\cE) \smile \pi_*c) \frown [X] + (-1)^{k - 1}\rk(\cE)l(c)\delta_k\varepsilon_{n - k},
\end{equation}
and as $(-1)^k\delta_k\varepsilon_{n - k} > 0,$
\begin{equation}
    \delta_k\tilde{P}_{\varepsilon,c}(\pi^*\cE) = \delta_kP_{\pi_*c}(\cE) + \delta_k\varepsilon_n(\Chern_k(\cE) \smile \pi_*c) \frown [X] - \rk(\cE)l(c)\left\lvert\varepsilon_{n - k}\right\rvert.
\end{equation}
Together with~\eqref{EQ:Borne l(c)},~\eqref{EQ:Borne P_c(E)}, and~\eqref{EQ:Borne Ch smile c} this gives
\begin{equation}
    \begin{split}
        \delta_k\tilde{P}_{\varepsilon,c}(\pi^*\cE) & \geq \epsilon\|\pi_*c\| - C'|\varepsilon_n|\|\pi_*c\| - C\rk(\cE)|\varepsilon_{n - k}|\|\pi_*c\| \\
        & = (\epsilon - C'\|\varepsilon_n\| - C\rk(\cE)|\varepsilon_{n - k}|)\|\pi_*c\|.
    \end{split}
\end{equation}
Let us set $\eta = \frac{\epsilon}{2(C' + C\rk(\cE))} > 0$. If $|\varepsilon_{n - k}| \leq \eta$, $|\varepsilon_n| \leq \eta$ and $\pi_*c \neq 0$, then
$$
\delta_k\tilde{P}_{\varepsilon,c}(\pi^*\cE) \geq \frac{\epsilon}{2}\|\pi_*c\| > 0.
$$
If $\pi_*c = 0$ and $c \neq 0$,~\eqref{EQ:Borne l(c)} implies that $l(c) < 0$. In this case, by Lemma~\ref{LEM:Calcul P tilde},
$$
\delta_k\tilde{P}_{\varepsilon,c}(\pi^*\cE) = (-1)^k\rk(\cE)\delta_k\varepsilon_{n - k} > 0.
$$
Hence the uniform $\tilde{P}_{\varepsilon,\delta}$-positivity of $\pi^*\cE$.

By Lemma~\ref{LEM:P positivité forte}, we easily show that the theorem holds for strong $P_\delta$-positivity too as long as $|\varepsilon_n| < \gamma_n$.
\end{proof}

\section{Applications}
\label{sec:applications}

\subsection{Toric examples of $Z$-positive bundles}
\label{sec:toric applications}
We will now prove the following propositions, as applications of Theorem~\ref{theo:intro equiv} and the discussion in Section~\ref{sec:toric case} (in particular Section~\ref{sec:slopes}):

\begin{proposition}
 \label{prop:introhirzebruchdHYM}
 Let $X$ be a Hirzebruch surface and let $\cE=-K_X$ or $\cE=\cT_X$ be its anti-canonical bundle or its tangent bundle. Then there is an ample line bundle $L$ on $X$ such that for the central charge $Z = (L, \rho, U)$ of the dHYM equation, $\cE$ is $Z$-positive.
\end{proposition}
We also give a $3$-dimensional example: 
\begin{proposition}
 \label{prop:introFano3fold}
 Consider the toric Fano $3$-fold $X = \P \left( \cO_{\P^1 \times \P^1} \oplus \cO_{\P^1 \times \P^1}(1,1) \right)$. There is an ample line bundle $L$ on $X$ such that for the central charge $Z = (L, \rho, U)$ of the dHYM equation, $\cT_X$ is $Z$-positive.
\end{proposition}

\subsubsection{Hirzebruch surfaces}
\label{sec:CanHirz}

 Assume that $X$ is a smooth projective toric surface. Then, for any toric sheaf $\cE$ over $X$, we have
\begin{equation}\label{eq:Zslopestoricsurace}
\begin{array}{ccc}
    Z_X(\cE) & = & \rank(\cE)(\rho_0U_2+\rho_1U_1\cdot L+\rho_2 L^2)\\
    & & +(\rho_0 U_1+\rho_1 L)\cdot \mathrm{ch}_1(\cE)+\rho_0 \mathrm{ch}_2(\cE).
\end{array}
\end{equation}
For any invariant curve $V\subset X$ we have
\begin{equation}
 \label{eq:Zslopesrestricedsurface}
Z_V(\cE)=\rho_1 \rank(\cE)(L\cdot V )+\rho_0\rank(\cE)( U_1\cdot V ) + \rho_0c_1(\cE)\cdot V.
\end{equation}
Recall the central charge of the dHYM equation:
\begin{equation}
 \label{eq:centralchargedHYM}
Z_X(\cE)=-\int_X e^{-iL}\cdot \mathrm{ch}(\cE). 
\end{equation}
Using equation~\eqref{eq:polynomialchargebundles}, for this central charge on surfaces we have
$$
U = 1 \quad \text{and} \quad (\rho_0, \rho_1, \rho_2) = (-1, i, 1/2).
$$
Let $r\in\N$ and consider the $r$-th Hirzebruch surface
$$X = \P(\cO_{\P^1} \oplus \cO_{\P^1}(r))$$
with the projection map $\pi\colon X \rightarrow \P^1$.
We set $F= \pi^* \cO_{\P^1}(1)$ and $H= \cO_X(1)$.
The rays of the fan of $X$ are the half lines generated by the vectors
$u_1 = e_1$, $u_2 = e_2$, $u_3 = -e_1 + r \, e_2$ and $u_0 = -e_2$ where $(e_1, e_2)$ is
the standard basis of $\Z^2$.
%
%
Let $D_i$ be the divisor associated to the ray $\Cone(u_i)$, we have
$$
D_1 \sim_{\rm lin} D_3 \sim_{\rm lin} F, \quad D_0 \sim_{\rm lin} H \quad
\text{and} \quad D_2 \sim_{\rm lin} H - rF.
$$
By \cite[Lemma 15.1.8]{CLS}, the effective cone of $X$ is generated by $F$ and $H-rF$.
We have the following intersection products (c.f.\ \cite[Example 6.4.6]{CLS}):
$$
F^2 = 0, \quad F \cdot H = 1 \quad \text{and} \quad H^2 = r.
$$
As the anti-canonical divisor of $X$ is 
$$
-K_X = D_0 + D_1 + D_2 + D_3 \sim_{\rm lin} (2-r)F + 2H,
$$
we get the following Chern character
$$
\ch(\cO_X(-K_X)) = 1 +(2-r)F + 2H + 4 F \cdot H.
$$
According to \cite[Proposition 4.2.1]{DDK} a divisor $L=aF + bH$ of $X$ is
ample (resp. nef) if and only if $a,b >0$ (resp. $a, b \geq 0$).

\begin{proof}[Proof of Proposition~\ref{prop:introhirzebruchdHYM}]
We set $\cE = \cO_X(-K_X)$ and let $L=a F + b H$ be an ample line bundle where $a,b \in \N^*$.
With~\eqref{eq:Zslopestoricsurace}, we obtain
$$
Z_X(\cE) = \left(ab + \frac{rb^2}{2} -4 \right) + i (2a + 2b + br),
$$
and in particular $\Im(Z_X(\cE)) > 0$. To get the $Z$-positivity of $\cE$ it suffices to compute
$\Im\left(\frac{Z_V(\cE)}{Z_X(\cE)}\right)$ for $V = F$ and $V=H-rF$.
By~\eqref{eq:Zslopesrestricedsurface}, we have $Z_F(\cE)  =   -2 + i b$
and $Z_{H-rF}(\cE) = -(2-r) + ia$.
Then we compute
\begin{align*}
\Im( Z_F(\cE) \overline{Z_X(\cE)}) & = 4a + a b^2 + \left( 2b + \frac{b^3}{2} \right)r
\end{align*}
and
\begin{align*}
\Im( Z_{H-rF}(\cE) \overline{Z_X(\cE)}) & =  -2ar + (a^2 - r^2 + 4)b + \dfrac{ar}{2} b^2.
\end{align*}
If $r=0$, then $\cE$ is $Z$-positive with respect to any $L$. For $r \geq 1$, we find that the
line bundle $\cE$ is $Z$-positive if and only if
\begin{equation}
    b > \sqrt{\frac{(a^2 - r^2 + 4)^2}{a^2 r^2} + 4} - \frac{a^2 - r^2 + 4}{ar}
\end{equation}
that can be satisfied for suitable $(a,b)$.

\smallskip

We then consider the tangent bundle. According to Example~\ref{ex:chern-of-tangent}, we have
$$
\ch(\cT_X) = 2 + (2-r)F + 2H.
$$
Let $L= aF + b H$ be an ample divisor. By~\eqref{eq:Zslopestoricsurace} We have
\begin{align*}
Z_X(\cT_X) & = (aF + bH)^2 + i (aF + bH) \cdot ((2-r)F + 2H)
\\ & = (2ab + r b^2) + i(2a+2b+br).
\end{align*}
On the other hand by~\eqref{eq:Zslopesrestricedsurface}, we have
\begin{align*}
Z_F(\cT_X) & = 2i(aF + bH)\cdot F - ((2-r)F + 2H) \cdot F\\
 & = -2 + 2i b
\end{align*}
and
\begin{align*}
Z_{H-rF}(\cT_X) & = 2i(aF + bH)\cdot (H -rF) - ((2-r)F + 2H) \cdot (H- rF)\\
 & = -(2-r) + 2ia.
\end{align*}
Then:
\begin{align*}
\Im( Z_F(\cT_X) \overline{Z_X(\cT_X)}) & = 2(2a+2b+br) + 2b(2ab + rb^2)
\\ & = 4(a + b + ab^2) + (2b + 2b^3)r
\end{align*}
and
\begin{align*}
\Im( Z_{H-rF}(\cT_X) \overline{Z_X(\cT_X)}) & = (2-r)(2a+2b+br) + 2a(2ab + rb^2)
\\ & = (4-2r)a + (4a^2 + 4 - r^2)b + 2arb^2.
\end{align*}
If $r \in \{0,1,2\}$, then $\cT_X$ is $Z$-positive with respect to any $L$.
For $r \geq 3$, the tangent bundle $\cT_X$ is $Z$-positive if and only if
$$
b > \dfrac{1}{2} \sqrt{\dfrac{(4a^2 + 4 - r^2)^2}{(2ar)^2} + \dfrac{4(r-2)}{r}}
- \dfrac{(4a^2 + 4 - r^2)}{4 ar}
$$
and we obtain the result for suitable choices of $(a,b)$.
\end{proof}

\subsubsection{A toric Fano \texorpdfstring{$3$}{3}-fold}
We conclude this section by proving Proposition~\ref{prop:introFano3fold}.
Let $\Sigma$ be a complete fan in $\R^3$ (c.f.\  Figure~\ref{fig:3fold}) with rays generated by
\begin{align*}
u_1 = (1,0,0) & & u_2= (0,1,0) & & u_3 = (-1,0,0) \\
u_4 = (1,-1,0) & & u_5 = (0,0,1) & & u_6= (1,0,-1).
\end{align*}
By \cite[Section 3]{Watanabe82}, the toric variety associated to the fan $\Sigma$ is
$$X = \P \left( \cO_{\P^1 \times \P^1} \oplus \cO_{\P^1 \times \P^1}(1,1) \right).$$
We denote by $D_i$ the divisor associated to the ray $\Cone(u_i)$. We have:
$$
D_4 \sim_{\rm lin} D_2 \qquad D_6 \sim_{\rm lin} D_5 \qquad
D_1 \sim_{\rm lin} D_3 - D_2 - D_5
$$
and 
$$-K_X \sim 2 D_3 + D_2 + D_5.$$

\begin{figure}
\begin{tikzpicture}[scale=0.8]
\draw[dashed] (-4.5,-2) -- (-1.5,2) -- (4.5,2) -- (1.5,-2) -- (-4.5,-2);
\draw[dashed] (-1.5,-5) -- (-1.5,1) -- (1.5,5) -- (1.5,-1) -- (-1.5,-5);

\draw[line width=0.8mm] (0,0) -- (0,3);
\draw[line width=0.8mm] (0,0) -- (3,0);
\draw[line width=0.8mm] (0,0) -- (-1.5,-2);
\draw[line width=0.8mm] (0,0) -- (-4.5,-2);
\draw[line width=0.8mm] (0,0) -- (-1.5,-5);
\draw[line width=0.8mm] (0,0) -- (1.5,2);

\draw[dashed] (0,0) -- (-3,0);
\draw[dashed] (0,0) -- (0,-3);

\draw (-1.5,-2) node{$\bullet$} node[below left]{$u_1$};
\draw (3,0) node{$\bullet$} node[right]{$u_2$};
\draw (1.5,2) node{$\bullet$} node[above right]{$u_3$};
\draw (-4.5,-2) node{$\bullet$} node[below]{$u_4$};
\draw (0,3) node{$\bullet$} node[above left]{$u_5$};
\draw (-1.5,-5) node{$\bullet$} node[below]{$u_6$};
\end{tikzpicture}

\caption{Fan in $\R^3$}
\label{fig:3fold}
\end{figure}

\begin{lemma}
We have the following intersections in the Chow ring:
$$
D_{1}^2 = -D_1 \cdot D_2 - D_1 \cdot D_5,
\qquad D_{3}^2 = D_1 \cdot D_2 + D_1 \cdot D_5 + 2 D_2 \cdot D_5,
$$
and
$$
D_{2}^2 = D_{5}^2 = 0, \qquad D_{1}^3 = D_{3}^3 = 2 D_1 \cdot D_2 \cdot D_5.
$$
\end{lemma}

\begin{proof}
As $D_2 \sim_{\rm lin} D_4$ and $\Cone(u_2, u_4) \notin \Sigma$, we deduce that 
$$D_{2}^2=0.$$
We obtain similarly
$$D_{5}^2 = D_5 \cdot D_6 = 0.
$$
For the divisor $D_1$, we have
\begin{align*}
D_{1}^2  &= D_1 \cdot (D_3 - D_2 - D_5) \\
 & = - D_1 \cdot D_2 - D_1 \cdot D_5.
\end{align*}
The other equalities follow from the linear equivalences of the divisors.
\end{proof}

\begin{lemma}
We have
$$
(-K_X)^2 \cdot D_1 = 2, \qquad (-K_X)^2 \cdot D_2 = 8, \qquad (-K_X)^2 \cdot D_5 = 8.
$$
\end{lemma}

\begin{proof}
Writing $-K_X \sim_{\rm lin} 3D_2 + 3D_5 + 2 D_1$, we obtain
$$
(-K_X)^2 = 18 D_2 \cdot D_5 + 8 D_2 \cdot D_1 + 8 D_1 \cdot D_5.
$$
Using $D_{2}^2=0$ and $D_{5}^2 = 0$, we get the second and the third equalities. As
$$D_{1}^2 = -D_1 \cdot D_2 - D_1 \cdot D_5,$$ we have
$$
D_2 \cdot D_{1}^2 = -1 \quad \text{and} \quad D_5 \cdot D_{1}^2 = -1.
$$
Hence, we have $(-K_X)^2 \cdot D_1 = 2$.
\end{proof}

We can now give the proof of Proposition~\ref{prop:introFano3fold}.

\begin{proof}[Proof of Proposition~\ref{prop:introFano3fold}]
By Example~\ref{ex:chern-of-tangent}, we have
$$
\ch(\cT_X) = 3 - K_X + D_2 \cdot D_5 + \dfrac{2}{3}D_1 \cdot D_2 \cdot D_5
$$
and
\begin{align*}
-e^{i K_X} \cdot \ch(\cT_X) = &
-\left(1 + i K_X - \dfrac{1}{2} K_{X}^2 - \dfrac{i}{6} K_{X}^3 \right) \cdot \ch(\cT_X)
\\ = &
-3 + (1 - 3i)K_X + \left(\dfrac{3}{2} K_{X}^2 - D_2 \cdot D_5 + i K_{X}^2 \right)
\\ & -
\left( \dfrac{2}{3} D_1 \cdot D_2 \cdot D_5 + \dfrac{1}{2} K_{X}^3 + i K_X \cdot D_2 \cdot D_5 - \dfrac{i}{2} K_{X}^3 \right).
\end{align*}
For $\lbrace i,j\rbrace\in \lbrace 1,\ldots, 6\rbrace^2$ with $i \neq j$, we set
$$
V_{ij} = D_i \cdot D_j.
$$
As any reduced invariant divisor $D$ is linearly equivalent to a non-negative combination
of $D_1$, $D_2$ and $D_5$, we deduce that for any $i,j \in \{1, \ldots,6\}$ with $i \neq j$, $V_{ij}$ can be written as a non-negative combination of $V_{12}$, $V_{15}, V_{25}$.
Therefore, to check the $Z$-positivity, it is enough to compute
$\Im\left( Z_V(\cT_X) \overline{Z_X(\cT_X)} \right)$ for
$$
V \in \{D_1, D_2, D_5, V_{12}, V_{15}, V_{25}\}.
$$
We then derive the following: 
\begin{align*}
Z_X(\cT_X) & = \frac{76}{3} - 24i,
\\
Z_{D_1}(\cT_X) & = \left(\dfrac{3}{2} K_{X}^2 - D_2 \cdot D_5 + i K_{X}^2 \right) \cdot D_1
= 2 + 2i,
\\
Z_{D_2}(\cT_X) & = \left(\dfrac{3}{2} K_{X}^2 - D_2 \cdot D_5 + i K_{X}^2 \right) \cdot D_2
= 12 + 8i,
\\
Z_{V_{12}}(\cT_X) & = -(1-3i)(-K_X \cdot D_1 \cdot D_2)= -1 + 3i,
\\
Z_{V_{25}}(\cT_X) &= -(1-3i)(-K_X \cdot D_2 \cdot D_5)= 2(-1 + 3i).
\end{align*}
According to the geometry of the fan, one has $$Z_{D_2}(\cT_X) = Z_{D_5}(\cT_X)$$ and
$$Z_{V_{12}}(\cT_X)= Z_{V_{15}}(\cT_X).$$ As for any $V \in \{D_1, D_2, D_5, V_{12}, V_{15}, V_{25}\}$,
$$
\Im\left(Z_V(\cT_X) \overline{Z_X(\cT_X)} \right) > 0,
$$
we deduce that $\cT_X$ is $Z$-positive.
\end{proof}

\subsection{A slope unstable asymptotically $P$-stable bundle}
\label{sec:application asymptotic case}
In  \cite[Example $3.5$]{KScarpa}, examples of $P$-stable bundles that are not slope polystable were provided on surfaces. In this section, we provide an example of a $P$-stable tangent bundle that is not slope polystable on a $3$-fold. As a corollary of Theorem~\ref{thm:DMS_polynomial}, this bundle will carry solutions to the $P$-critical equations in asymptotic regimes, for a polynomial $P$ coming from the $\alpha$-stability notion introduced in \cite{Toma_semistab_ampleclasses} (c.f.\ Section~\ref{sec:asymptotic Toma et al}).
The proof is an application of Theorem~\ref{theo:intro equiv Zstab asymptotic}, noting that such a polynomial is adapted to torsion-free sheaves (c.f.\ \cite[Example 2.4]{Delloque_adapted}).

Let $X= \P \left( \cO_{\P^2} \oplus \cO_{\P^2}(1) \right)$ with the projection map
$\pi\colon X \rightarrow \P^2$. We set $H = \cO_X(1)$ and $F= \pi^* \cO_{\P^2}(1)$.
The rays of the fan of $X$ are the half lines generated by the vectors
$u_1 = e_1$, $u_2 = e_2$, $u_3 = e_3$, $u_4 = -e_3$ and $u_0 = e_3 -e_1 -e_2$
where $(e_1, e_2, e_3)$ is the standard basis of $\Z^3$.
If $D_i$ be the divisor associated to the ray $\Cone(u_i)$, we have
$$
D_0 \sim_{\rm lin} D_1 \sim_{\rm lin} D_2 \sim_{\rm lin} F, \quad D_4 \sim_{\rm lin} H
\quad \text{and} \quad
D_3 \sim_{\rm lin} H - F.
$$
We also have the following intersection products:
$$
H^3 = H^2 \cdot F = H \cdot F^2 = 1 \quad \text{and} \quad F^3 = 0.
$$
The Chern character of the tangent bundle is
$$
\ch(\cT_X) = 3 + 2(F+H) + 2F^2 - HF + H^2
$$
and the Todd Class of $X$ is given by
$$
\Todd_X = 1 + (F+H) + \dfrac{1}{24}(10 F^2 + 25 FH + 11 H^2) + \dfrac{13}{12} HF^2.
$$
According to \cite[Theorem 1.4]{HNS}, the tangent bundle $\cT_X$ is not stable with
respect to $L= H + F$. Applying Lemma~\ref{lem:sat subsheaves equiv case} and Examples~\ref{ex:tangent} and~\ref{ex:line bundles}, we find that the saturated equivariant subsheaves involved in the study of
stability of $\cT_X$ are $\cO_X(D_0)$, $\cO_X(D_1)$, $\cO_X(D_2)$, $\cO_X(D_3 + D_4)$ and the
reflexive sheaves of rank $2$ obtained as the direct sums of the previous.

We now study asymptotically $P$-stability of $\cT_X$ using the
$P$-condition given in equation~\eqref{eq:Pcondition-Toma}. We assume that
$$
(\alpha_0, \alpha_1, \alpha_2, \alpha_3) = ( [X], L, H F, H^2 F).
$$
We set set $\cF= \cO_X(F)$ and $\cG= \cO_X(2H-F)$. We have
$$
\ch(\cF) = 1 + F + \dfrac{1}{2} F^2 \quad \text{and} \quad
\ch(\cG) = 1 + (2H-F) + H^2 - HF + \dfrac{1}{2} F^2.
$$
Therefore,
\begin{align*}
P_{\alpha}(\cE, m) &= \dfrac{1}{2} m^3 + \dfrac{10}{2}m^2 + ...
\\
P_{\alpha}(\cF, m) &= \dfrac{1}{6} m^3 + \dfrac{3}{2} m^2 + ...
\\
P_{\alpha}(\cG, m) &= \dfrac{1}{6} m^3 + \dfrac{3}{2} m^2 + ...
\end{align*}
For large $m$, we have $P_{\alpha}(\cF, m) \leq \frac{1}{3} P_{\alpha}(\cE, m)$ and
$P_{\alpha}(\cG, m) \leq \frac{1}{3} P_{\alpha}(\cE, m)$. Thus $\cE$ is asymptotically
$P$-stable.

\bibliographystyle{amsplain}
\bibliography{Ztoric}

\end{document}